\documentclass[final,leqno,onefignum,onetabnum]{siamltex1213}
\usepackage[english]{babel}
\usepackage{bm}
\usepackage{mathtools}
\usepackage{amssymb}
\usepackage{amsfonts}
\usepackage{stmaryrd}
\usepackage{epstopdf}
\usepackage{mathrsfs}
\usepackage{cleveref}
\usepackage{multirow}

\newcommand{\enorm}[1]{{\left\vert\kern-0.25ex\left\vert\kern-0.25ex\left\vert #1 
  \right\vert\kern-0.25ex\right\vert\kern-0.25ex\right\vert}}
 \newtheorem{remark}{Remark}[section]

\begin{document}
\title{Robust discretization of flow in \\ fractured porous media}
\author{Wietse M. Boon\footnotemark[2] \and Jan M. Nordbotten\footnotemark[2]\ \footnotemark[3] \and Ivan Yotov \footnotemark[4]}

\maketitle

\renewcommand{\thefootnote}{\fnsymbol{footnote}}
\footnotetext[2]{Department of Mathematics, University of Bergen, Postbox 7803, 5020 Bergen, Norway (\email{wietse.boon@uib.no}, \email{jan.nordbotten@uib.no}) This work was supported in part by Norwegian Research Council grants 233736 and 228832.}
\footnotetext[3]{Department of Civil and Environmental Engineering, Princeton University, Princeton, NJ 08544, USA.}
\footnotetext[4]{Department of Mathematics, University of Pittsburgh, Pittsburgh, PA, USA.}
\renewcommand{\thefootnote}{\arabic{footnote}}

\begin{abstract}
	Flow in fractured porous media represents a challenge for discretization methods due to the disparate scales and complex geometry. Herein we propose a new discretization, based on the mixed finite element method and mortar methods. Our formulation is novel in that it employs the normal fluxes as the mortar variable within the mixed finite element framework, resulting in a formulation that couples the flow in the fractures with the surrounding domain with a strong notion of mass conservation. The proposed discretization handles complex, non-matching grids, and allows for fracture intersections and termination in a natural way, as well as spatially varying apertures. The discretization is applicable to both two and three spatial dimensions. A priori analysis shows the method to be optimally convergent with respect to the chosen mixed finite element spaces, which is sustained by numerical examples.
\end{abstract}

\pagestyle{myheadings}
\thispagestyle{plain}
\markboth{W. M. Boon, J. M. Nordbotten and I. Yotov}{ROBUST DISCRETIZATION OF FLOW IN FRACTURED POROUS MEDIA}

\section{Introduction} 
\label{sec:introduction}

	Fractures are ubiquitous in natural rocks, and in many cases have a leading order impact on the structure of fluid flow \cite{adler1999fractures,Helmig}. Due to great differences in permeability, the fractures may either conduct the flow or act as blocking features. Due to their significant impact, detailed and robust modeling of coupled flow between fractures and a permeable rock is essential in applications spanning from enhanced geothermal systems, to CO$_2$ storage and petroleum extraction. 

	Because of the complex structure of natural fracture networks \cite{Helmig}, it remains a challenge to provide robust and flexible discretization methods. Here, we identify a few distinct features which are attractive from the perspective of applications. The method formulated in this work is specifically designed to meet these goals.

	First, we emphasize the importance of mass conservative discretizations. This is of particular significance when the flow field is coupled to transport (of e.g. heat or composition), as transport schemes are typically very sensitive to non-conservative flow fields \cite{levequeFVM}. 
	The second property of interest is grid flexibility. This is important both in order to accommodate the structure of the fracture network, but also in order to honor other properties of the problem, such as material heterogeneities or anthropogenic features such as wells \cite{mallisongridding}.
	Third, it is necessary that discretization methods are robust in the physically relevant limits. In the case of fractures, it is imperative to allow for arbitrarily large aspect ratios, that is to say, thin fractures with arbitrarily small apertures, including the aperture going to zero as fractures terminate.
	Finally, our interest is in provably stable and convergent methods. 


	Since their aspect ratios frequently range as high as 100-1000, it is appealing to consider fractures as lower-dimensional features, as was first explored in \cite{alboin1999,alboin2002modeling}. In this setting, we consider a three-dimensional domain of permeable rock, within which (multiple) fractures will be represented by (multiple) two-dimensional manifolds. In the case where two or more fractures intersect, we will naturally also be interested in the intersection lines and points. Our approach handles such manifolds, lines, and points in a unified manner.

	Several methods have been proposed to discretize fractured porous media, some of which are reviewed below. However, to our knowledge, no method has been presented which fulfills the four design goals outlined above. 

	A natural discretization approach to obtain conservative discretizations is to consider finite volume methods adapted to fracture networks (see e.g. \cite{karimi2004,Sandve2}). Here, the fractures are added as hybrid cells between the matrix cells. The small cells which are formed at the intersections are then excluded with the use of transformations in order to save condition numbers and computational cost. However, the formulation requires the grids to match along the fractures. The incorporation of non-matching grids along faults was analyzed by Tunc et al. \cite{Tunc}. While the presented finite volume formulations are formally consistent methods, convergence analyses of these methods are lacking. 

	Alternatively, the extended finite element (XFEM) approach \cite{Scotti,Formaggia,Schwenck2015} is a method in which the surroundings are meshed independently from the fractures. The fracture meshes are then added afterwards, crossing through the domain and cutting the elements. Although this may be attractive from a meshing perspective, the cut elements may become arbitrarily small such that special constructions are needed to ensure stability. Such constructions are typically introduced whenever multiple fractures, intersections, and fracture endings are considered in the model. Our aim is to develop a method with a unified approach to such features and a different approach is therefore chosen. Admittedly, the construction of meshes will be more involved for complicated cases but we aim to relieve this by allowing for non-matching grids.

	The Mixed Finite Element (MFE) method \cite{Boffi,brezzi1991mixed} is employed in this work, since it provides two important advantages. The method defines the flux as a separate variable and mass conservation can therefore be imposed locally. Furthermore, the tools necessary to perform rigorous analysis can be adapted from those available in the literature.

	Mortar methods, as introduced in \cite{bernardi}, form an appealing framework for fracture modeling, since both non-matching grids and intersections are naturally handled. The combination with MFE has since been explored extensively (see e.g. \cite{Arbogast,pencheva2003balancing}). The idea of conductive fractures was first exploited in \cite{Roberts,Roberts2}, where Darcy flow is allowed inside the mortar space based on the pressure variable. However, in previously developed mortar MFE methods, the choice of using the pressure variable in the mortar space does not allow for strong flux continuity. 

	Herein we propose a new method, based on the structure of mortar MFE methods. Our formulation is novel in that it employs the fluxes as the mortar variable within the mixed finite element framework. Thus, the proposed method couples the flow in fractures with the surrounding domain using a stronger notion of mass conservation. For domain decomposition with matching grids, flux Lagrange multiplier for MFE methods was proposed in \cite{GW}. To the best of our knowledge, this technique has not been explored in the context of mortar MFE methods on non-matching grids. The method is designed with the four goals outlined above in mind. 

	We formulate the method hierarchically, which allows for a unified treatment of the permeable domain, the fractures, intersection lines, and intersection points in arbitrary dimensions. We show through rigorous analysis that the method is robust with respect to the aspect ratio, however we exclude the case of degenerate normal permeability from our analysis. The numerical results verify all the analytical results, and furthermore indicate stability also in the case of degenerate normal permeability. 

	The paper is organized as follows. Section~\ref{sec:model_formulation} introduces the model in a continuous setting and explains the concept of composite function spaces formed by function spaces with different dimensions. Section~\ref{sec:discretization} is devoted to the discretized problem and the analytical proofs of properties such as stability and convergence. Finally, results of numerical experiments confirming the theory in two and three dimensions are presented in Section~\ref{sec:numerical_results}. We point out that a full numerical comparison to the alternative discretization methods discussed above has been conducted separately as part of a benchmark study \cite{flemisch2017benchmarks}.

\section{Model Formulation} 
\label{sec:model_formulation}

	In this section, we first describe the notion of working with subdomains with different dimensions and introduce the notation used in this paper. Next, the governing equations for the continuous problem are derived and presented. The section is concluded with the derivation of the weak formulation of the problem.

\subsection{Geometric Representation} 
\label{sub:geometric_representation}

	Consider an $n$-dimensional domain $\Omega$, which is decomposed into subdomains with different dimensionalities.  Setting the ambient dimension of the problem $n$ equal to 2 or 3 will suffice for most practical purposes, but the theory allows for $n$ to be arbitrary. The subdomains of dimension $n-1$ then represent fractures, while the lower-dimensional domains represent intersection lines and points. 

	We start by establishing notation. Let $N^d$ denote the total number of $d$-dimensional subdomains and let each open, $d$-dimensional subdomain be denoted by $\Omega^d_i$ with $0 \leq d \leq n$ and counting index ${i \in \{ 1,2, \dots N^d\}}$. For notational simplicity, the union of all $d$-dimensional subdomains is denoted by $\Omega^d$:
	\begin{align*}
		\Omega^d &= \bigcup_{i=1}^{N^d} \Omega^d_i.
	\end{align*}
	A key concept in the decomposition is that all intersections of $d$-dimensional subdomains are considered as $(d-1)$-dimensional domains. In turn, the domain $\Omega^{d-1}$ is excluded from $\Omega^d$. For example, the point at the intersection between two lines becomes a new, lower-dimensional subdomain $\Omega^0$ which is removed from $\Omega^1$. An illustration of the decomposition in two dimensions is given in Figure~\ref{fig: Domain_decomp} (Left). The procedural decomposition by dimension applies equally well to problems in three dimensions.

	\begin{figure}[thbp]
		\centering
		\includegraphics[width = 0.9\linewidth]{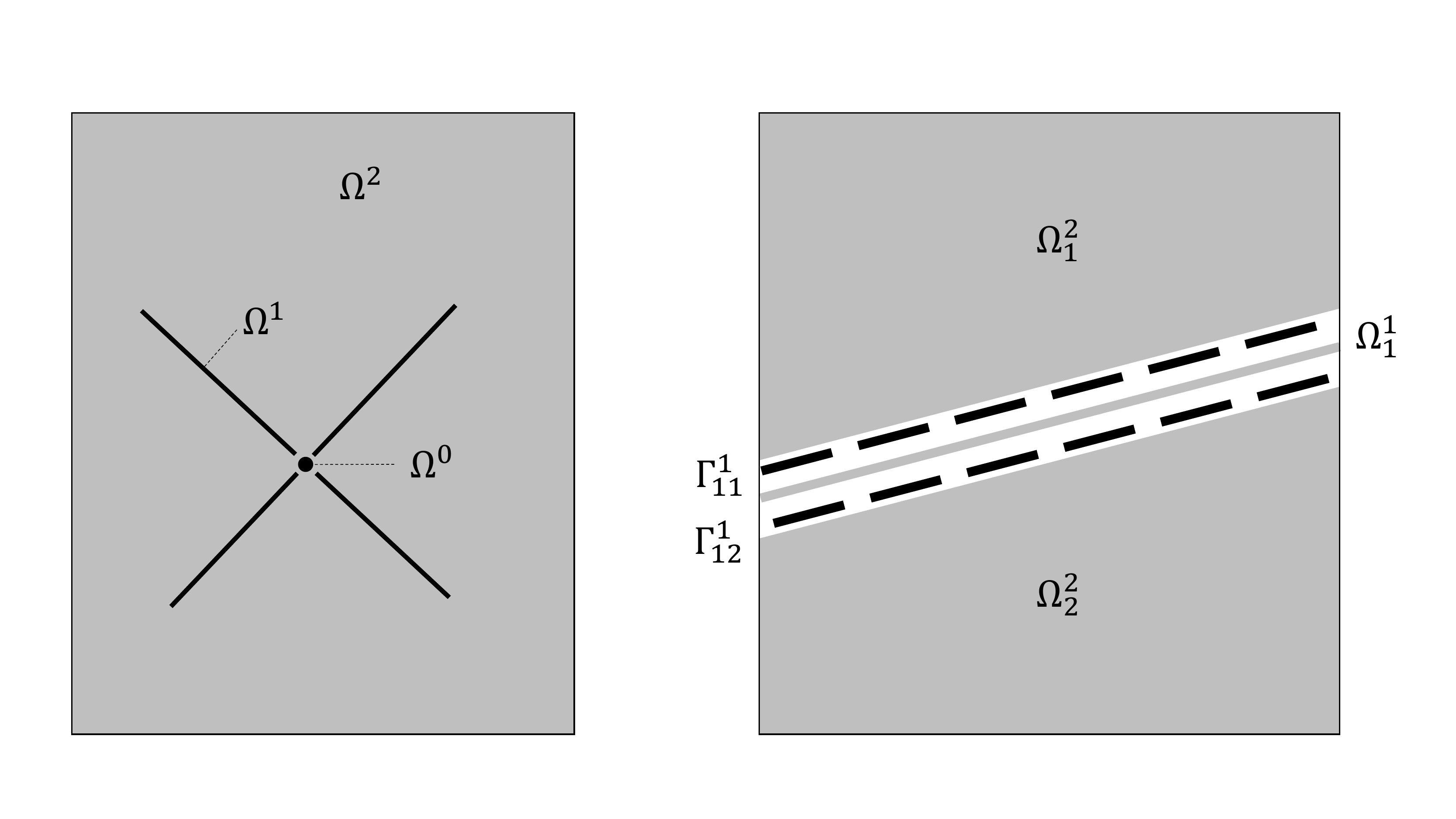}
		\caption{(Left) The domain is decomposed into subdomains where the dimensionality of each subdomain is given by the superscript. This decomposition allows us to model fractures and intersections as lower-dimensional features in the domain. In this particular illustration, we have four fracture segments, thus $N^1 = 4$. (Right) The interface $\Gamma^1$ in case of a single fracture. We define $\Gamma$ as the union of interfaces between domains of codimension one.}
		\label{fig: Domain_decomp}
	\end{figure}

	Physically, the flow between domains of different dimension (e.g. between fracture and matrix) is of particular interest. We are thus careful with the interfaces between subdomains of successive dimension. For each subdomain $\Omega^d_i$ with $d \le n - 1$, we define $\mathcal{J}_i^d$ as a set of local counting indices which enumerates its adjacent $d$-interfaces. In turn, each interface is denoted by $\Gamma_{ij}^d$ with $j \in \mathcal{J}_i^d$. Analogous to the notation as employed above, we define the following geometric entities as
	\begin{align*}
		\Gamma^d_i &= \bigcup_{j \in \mathcal{J}_i^d} \Gamma^d_{ij}, & \Gamma^d &= \bigcup_{i=1}^{N^d} \Gamma^d_i, & \Gamma &= \bigcup_{d=0}^{n-1} \Gamma^d.
	\end{align*}

	The interface $\Gamma^d_i$ coincides spatially with $\Omega^d_i$, but its importance lies in being a subset of the boundary of the adjacent $(d+1)$-dimensional domains. An illustration of $\Gamma^1$ in a two-dimensional setting is given in Figure~\ref{fig: Domain_decomp} (Right).

	At this point, we have the necessary entities to introduce the dimensional decomposition of the domain $\Omega$ and its boundary:
 	\begin{align}
 		\Omega \cup \partial \Omega &= \left(\bigcup_{d=0}^{n} \Omega^d \right) \cup \left(\bigcup_{d=1}^{n}\partial \Omega^d \backslash \Gamma^{d-1} \right).
 	\end{align}

	Let $\bm{\nu}$ denote the outward unit normal to $\Omega^d$, defined on $\partial \Omega^d$. By definition, $\bm{\nu}$ on $\Gamma^d$ is thus directed from $\Omega^{d+1}$ to $\Omega^d$, i.e. towards the lower-dimensional subdomain.

	The boundary of the model domain will enter naturally with the governing equations below. We emphasize that domains of any dimension may contact the domain boundary. Also, the case of subdomains with codimension two will not be considered in this work (e.g. line wells in 3D or two planar fractures meeting at a point). Nevertheless, it is possible to fit those cases into this framework by introducing specifically constructed subdomains of intermediate dimension.
	
	As a minor comment we note that the geometric representation, as well as much of the analysis below, can be generalized to calculus on manifolds. However, while the framework of manifolds does increase the mathematical elegance, and in some places simplifies and makes the exposition more precise, we believe that the current presentation is accessible to a wider readership. As an immediate consequence of this choice, we will from here on assume that all domains $\Omega^d_i$ are flat.

\subsection{Governing equations} 
\label{sub:governing_equations}

	The model considered in this work is governed by two physical relationships, namely mass conservation and Darcy's law. In particular, it is assumed that Darcy's law holds not just in the porous material, but also in all lower-dimensional subdomains. This corresponds to the physical situation of either thin open fractures (Poiseulle flow), or fractures filled with some material. The mathematical representations of these relationships have been well established and employed by several models \cite{Angot,Arbogast,Scotti,Roberts}. Here, we will introduce these relationships within the dimensional decomposition framework. Starting with the governing equations in the surrounding regions, we then continue with their analogues in lower-dimensional subdomains and finish with the coupling equations.

	First, let us consider the surroundings $\Omega^n$. We aim to find the flux $\bm{u}^n$ and pressure $p^n$ satisfying
	\begin{subequations}
		\begin{align}
			\bm{u}^n &= - K \nabla p^n &\text{in } \Omega^n, \label{eq: Darcy}\\
			\nabla \cdot \bm{u}^n &= f &\text{in } \Omega^n, \label{eq: massconv}\\
			p^n &= g &\text{on } \partial\Omega_D^n, \label{eq: omega_D}\\
			\bm{u}^n \cdot \bm{\nu} &= 0 &\text{on } \partial\Omega_N^n. \label{eq: omega_N}
		\end{align}
	\end{subequations}
	Here, we assume that the boundary of $\Omega$ can be parititioned as $\partial \Omega = \partial\Omega_D \cup \partial\Omega_N$, with $\partial\Omega_D \cap \partial\Omega_N = \emptyset$ and $\partial\Omega_D$ with positive measure. We assume that each subdomain $\Omega_i^n$ of $\Omega^n$ has a non-empty Dirichlet boundary, i.e. $|\partial \Omega_i^n \cap \partial \Omega_D| > 0$. The following notation is then employed within the dimensional decomposition framework:
	\begin{align*}
		\partial \Omega^d_D &= \partial \Omega^d \cap \partial \Omega_D, &
		\partial \Omega^d_N &= \partial \Omega^d \cap \partial \Omega_N, & 1 \le d \le n.
	\end{align*}

	Furthermore, $K$ is a bounded, symmetric, positive definite, $n \times n$ tensor representing the material permeability. Equation \eqref{eq: Darcy} is known as Darcy's law and equation \eqref{eq: massconv} is conservation of mass in the case of incompressible fluids.

	We continue with the governing equations defined on the lower-dimensional subdomains. In order to derive these equations with the correct scaling, two physical parameters are introduced, inherent to the geometry of the problem. First, on each $\Gamma_{ij}^d$, $0 \le d \le n-1$, let $\gamma_{ij}^d$ denote the length from $\Gamma_{ij}^d$ to the center of $\Omega^d_i$. For brevity, we will generally omit the indices on $\gamma$ and all other parameters
	
	Secondly, on each subdomain $\Omega^d_i$ with $0 \le d \le n-1$, let $\epsilon$ represent the square root of the cross-sectional length if $d = n - 1$, area if $d = n - 2$, or volume if $d = n - 3$. Ergo, $\epsilon$ scales as $\gamma^{\frac{n - d}{2}}$ by definition. We assume that both $\epsilon$ and $\gamma$ are bounded and known a priori and extended to the surroundings by setting $\epsilon = \gamma = 1$ in $\Omega^n$. 

	In general, we allow $\epsilon$ and $\gamma$ to vary spatially. As such, we are particularly interested in the case of closing fractures, i.e. where $\epsilon$ and $\gamma$ decrease to zero. Regarding the rate at which this is possible, we assume that the following holds almost everywhere in the sense of the Lebesgue measure:
	\begin{align}
		| \nabla \epsilon | \lesssim \epsilon^{\frac{1}{2}}, \label{grad eps bound}
	\end{align}
	with $|\cdot|$ denoting the Euclidean norm.

	Here, and onwards, the notation $a \lesssim b$ is used to imply that a constant $C > 0$ exists, independent of $\epsilon$, $\gamma$, and later $h$ such that $a \le Cb$. The relations $\gtrsim$ and $\eqsim$ have analogous meaning.
	
	The hat-notation $\hat{\epsilon}$ is used to denote the trace of $\epsilon$ onto $\Gamma_{ij}^d$ from one level higher, i.e. $\epsilon$ defined on $\Omega^{d+1}$:
	\begin{align*}
		\hat{\epsilon}_{ij}^d &:= \epsilon_{j}^{d+1}|_{\Gamma_{ij}^d}, & j &\in \mathcal{J}_i^d, &d &\le n-1.
	\end{align*}
	We set $\hat{\epsilon} = 1$ in $\Omega^n$. Due to the construction of the dimensional decomposition, we assume that $\Omega_i^d$ borders on a subdomain $\Omega_j^{d+1}$ with positive aperture for at least one index $j_{\max} \in \mathcal{J}_i^d$. The parameter $\hat{\epsilon}$ corresponding to this index is referred to as follows
	\begin{align}
		\hat{\epsilon}_{\max}(x) &:= \hat{\epsilon}_{i,j_{\max}}^d(x) > 0
		, & x &\in \Omega_i^d, & 0 &\le d \le n - 1. \label{eps_max bound}
	\end{align}

	The relationship between $\epsilon$ and $\hat{\epsilon}_{\max}$ is then assumed to satisfy
	\begin{align}
		\| \epsilon^{\frac{1}{2}} \|_{L^{\infty}(\Omega_i^d)} \| \hat{\epsilon}_{\max}^{-1} \|_{L^{\infty}(\Omega_i^d)} &\lesssim 1, & 1 &\le d \le n. \label{eps upperbounded by eps_max}
	\end{align}
	To justify this property, we derive for $d \ge 1$ and $n \le 3$ that
		\begin{align*}
			\hat{\epsilon}_{\max}^{-1} \epsilon^{\frac{1}{2}} 
			\eqsim \hat{\gamma}^{-\frac{n - (d+1)}{2}} \gamma^{\frac{n - d}{4}} 
			\lesssim \gamma^{- \frac{n - d}{4} + \frac{1}{2}}
			\lesssim \gamma^{0} = 1.
		\end{align*}
	Since this relationship will later be used for the fluxes in $\Omega^d$, \eqref{eps upperbounded by eps_max} is not necessarily imposed for $d = 0$.
	
	With the defined $\epsilon$, the scaled flux $\bm{u}^d$, $1 \le d \le n$, is introduced such that
	\begin{align}
		\bm{u}^d := \epsilon \tilde{\bm{u}}^d, \label{scaled fluxes}
	\end{align} 
	with $\tilde{\bm{u}}^d$ as the average, tangential flux in $\Omega^d$. In other words, $\bm{u}^d$ can be described as an intermediate definition between the average flux $\tilde{\bm{u}}^d$ and the integrated flux (given by $\epsilon \bm{u}^d$). It is reminiscent of the scaled flux presented in \cite{arbogast2017mixed}.

	In order to derive the conservation equation on a lower-dimensional surface, the fluxes entering through the boundary $\Gamma^d$ must be accounted for \cite{nordbottenbook}. Let $\lambda^d$, $0 \le d \le n-1$, denote $\bm{u}^{d + 1} \cdot \bm{\nu}$ on $\Gamma^d$. Here $\bm{\nu}$ is the normal vector associated with $\Gamma^d$ as defined in Subsection~\ref{sub:geometric_representation}. 

	Mass may enter the fracture from one side and continue tangentially through the fracture creating a (pointwise) difference in normal fluxes. To capture this jump, $\lambda^d$ will consist of multiple components $\lambda_{ij}^d$, each representing a scaled flux across $\Gamma_{ij}^d$. 

	Recall the set $\mathcal{J}_i^d$ of local indices at $\Omega^d_i$ as defined in Subsection~\ref{sub:geometric_representation}. The jump operator is then given by
	\begin{align}
		\llbracket \cdot \rrbracket &: L^2(\Gamma^d) \to L^2(\Omega^d), & 
		\llbracket \lambda \rrbracket|_{\Omega^d_i} &= -\sum_{j \in \mathcal{J}_i^d} \lambda_{ij}^d, & 
		0 \le d \le n-1. \label{jump}
	\end{align}

	The definitions introduced in this section allow us to deduce the mass conservation equation for the lower-dimensional domains. Let us consider $\Omega^1$ with $n = 2$ and integrate the mass conservation equation \eqref{eq: massconv} over a quadrilateral region $\omega$ illustrated in Figure~\ref{fig: Reduction}. 
	\begin{figure}[thbp]
		\centering
		\includegraphics[width = \linewidth]{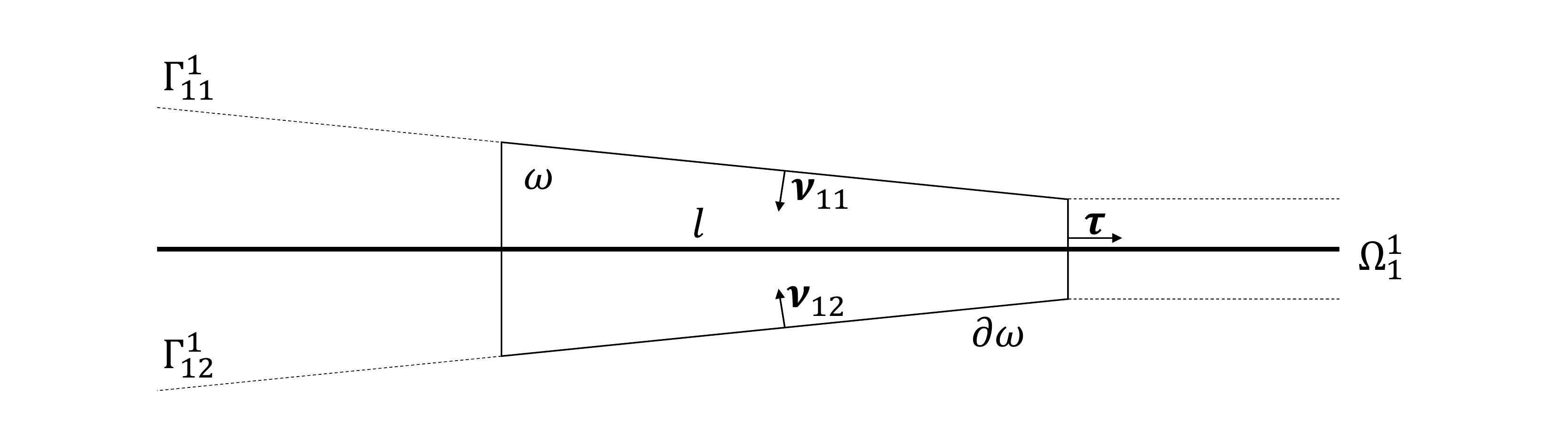}
		\caption{Local geometry for derivation of the conservation law. $\Omega^1$ represents the reduced, lower-dimensional manifold whereas the boundary between the fracture and matrix is given by $\Gamma$.}
		\label{fig: Reduction}
	\end{figure}

	We apply the divergence theorem on $\omega$ to derive
	\begin{align}
		\int_{\omega} \nabla \cdot \bm{u} = 
		\int_{\partial \omega_r} \bm{u} \cdot \bm{\tau}
		- \int_{\partial \omega_l} \bm{u} \cdot \bm{\tau}
		- \int_{\partial \omega \cap \Gamma} \bm{u} \cdot \bm{\nu}.
	\end{align}

	Next, we let the width of $\omega$, given by $l$, decrease to zero. The definition of the scaled fluxes from \eqref{scaled fluxes} and the factor $\hat{\epsilon} = 1$ gives us
	\begin{align}
		\lim_{l \to 0} l^{-1} \int_{\omega} \nabla \cdot \bm{u} &= 
		\nabla_{\tau} \cdot \epsilon \bm{u}^1 
		+ \big\llbracket \big(1 + \big|\nabla_{\tau} \frac{\gamma}{2}\big|^2 \big)^{\frac{1}{2}} \hat{\epsilon} \bm{u}^2 \cdot \bm{\nu} \big\rrbracket \nonumber\\
		&= \nabla_{\tau} \cdot \epsilon \bm{u}^1 
		+ \big\llbracket \big(1 + \big|\nabla_{\tau} \frac{\gamma}{2}\big|^2 \big)^{\frac{1}{2}} \hat{\epsilon} \lambda^1 \big\rrbracket,
	\end{align}
	with $\nabla_{\tau}$ the nabla operator tangential to $\Omega^1$.
	Note that the term $\big(1 + \big|\nabla_{\tau} \frac{\gamma}{2}\big|^2 \big)^{\frac{1}{2}}$ is close to unity since the changes in aperture are small by \eqref{grad eps bound}. We will therefore omit this factor for simplicity of exposition, while understanding that it can be subsumed into the definition of $\hat{\epsilon}$ at no additional theoretical complexity, and thus state the resulting conservation law as 
	\begin{align}
		\nabla \cdot \epsilon \bm{u}^d + \llbracket \hat{\epsilon} \lambda^d \rrbracket &=  \epsilon^2 f^d &\text{in } \Omega^d, \ 1 \le d \le n-1. \label{massconvlower}
	\end{align}	
	Here, $f^d$ represents the averaged source terms within $\Omega^d$. From here on, we denote $\nabla$ as the $d$-dimensional vector differential operator in $\Omega^d$. The case $d=0$ deserves additional attention since there is no tangential direction in which flow is possible. In turn, the mass conservation equation is reduced to 
	\begin{align}
		\llbracket \hat{\epsilon} \lambda^0 \rrbracket  &= \epsilon^2 f^0 &\text{in } \Omega^0. \label{Massconvlower 0}
	\end{align}	

	Equation \eqref{massconvlower} is simplified by introducing the semi-discrete differential operator $D$:
	\begin{align}
		D \cdot [\bm{u}^d, \lambda^d] &:= \nabla \cdot \bm{u}^d + \llbracket \lambda^d \rrbracket. \label{div D}
	\end{align}

	Continuing with the constitutive relationships, we consider Darcy's law in lower dimensions as described by the following linear expression:
	\begin{align}
		\epsilon^{-1} \bm{u}^d &= - K\nabla p^d &\text{in } \Omega^d, \ 1 \le d \le n-1. \label{Darcylower}
	\end{align}
	Note that we abuse notation once more by defining the permeability $K$ as a $d \times d$ tensor when used in $\Omega^d$. 

	The required boundary conditions for the lower-dimensional problems are chosen in the following way. First, the fracture may cross the domain and end on the boundary $\partial \Omega$. In that case, the imposed boundary condition in $\Omega^d$ is chosen to coincide with the boundary condition defined for the corresponding portion of $\partial \Omega$. In other words, if the fracture ends on $\partial \Omega_N$, a no-flux condition is imposed. On the other hand, if it ends on $\partial \Omega_D$, the pressure value is set to the average of $g$ across the cross section of $\Omega^d$, which we denote by $g^d$.
		\begin{align*}
			p^d &= g^d &\text{on } &\partial\Omega_D^d, \\
			\bm{u}^d \cdot \bm{\nu} &= 0 &\text{on } &\partial\Omega_N^d, & 1 &\le d \le n-1.
		\end{align*}
	The remainder of $\partial \Omega^d$ either borders on a lower-dimensional domain or represents an immersed tip. In the former case, a flux boundary condition is imposed on $\Gamma^{d-1}$ using the previously defined variable $\lambda^{d-1}$. In case of immersed tips, we assume that the mass transfer through the tip is negligible due to the large ratio between the fracture aperture and length. Therefore, in accordance with \cite{Angot}, a no-flux boundary condition is imposed. The boundary conditions are summarized as
	\begin{align*}
		\bm{u}^d \cdot \bm{\nu} &= \lambda^{d - 1} & \text{ on } \Gamma^{d-1 },&  \ 1 \le d \le n,\\
		\bm{u}^d \cdot \bm{\nu} &= 0 & \text{ on } \partial \Omega^d \backslash (\Gamma^{d-1 } \cup \partial \Omega),& \ 1 \le d \le n-1.
	\end{align*} 
	We will also allow for $\epsilon \downarrow 0$ at fracture tips, leading to a degenerate equation wherein the boundary condition is mute.

	Analogous to \cite{Angot,Roberts2}, Darcy's law is assumed to describe the flow normal to the fracture. For this, we introduce the normal permeability $K_{\nu}^d$ in $\Omega^d$ and impose the following relationship between the scaled, normal flux $\lambda^d$ and the pressure difference on $\Gamma_{ij}^d$:
	\begin{align}
		\hat{\epsilon}^{-1} \lambda_{ij}^d &= - K_{\nu} \frac{p_i^d - p^{d+1}|_{\Gamma_{ij}^d}}{\gamma}, & \ 0 \le d \le n-1, \label{pressure jump}
	\end{align}
	where we use the notation $p_i^d = p^d|_{\Omega_i^d}$. Moreover, sufficient regularity of $p$ is assumed in order to take such traces.

	The above represents the full description of the model equations considered herein, and is the setting in which the numerical method is constructed and validated. However, the analysis of both the continuous and discrete settings is restricted to the case where we have two further constants $c_0$ and $c_1$ such that the normal permeability is not degenerate in the following sense:
	\begin{align}
		0 < c_0 \le \gamma K_{\nu}^{-1} \le c_1 < \infty, \label{gamma_0}
	\end{align}
	similar to \cite{Roberts2}. We note in particular that the lower bound is needed for the completeness of the solution space under the chosen norms, see Lemma~\ref{Lem: Completeness}.

	The above equations comprise our model problem for flow in fractured porous media. 
	
\subsection{Weak Formulation} 
\label{sub:weak_formulation}

	Let us continue by deriving the weak formulation of the problem. For this, we introduce the function spaces associated with the dimensional decomposition introduced in Subsection~\ref{sub:geometric_representation}. For each value of $d$ denoting the dimensionality, let the function space $\bm{V}^d$ contain the (tangential) flux, let $\Lambda^d$ contain the flux across subdomain interfaces, and let $Q^d$ contain the pressure. For the continuous weak formulation, we define these function spaces as
	\begin{align*}
			\bm{V}^d &= \big\{\bm{v} \in (L^2(\Omega^d))^d: \nabla \cdot \epsilon \bm{v} \in L^2(\Omega^d)
			,\ (\epsilon \bm{v} \cdot \bm{\nu})|_{\partial \Omega^d \backslash (\Gamma^{d-1} \cup \partial \Omega_D)} = 0 \big\}, & 1 &\leq d \leq n,\\
			\Lambda^d &= L^2(\Gamma^d),  & 0 &\leq d \leq n-1,\\
			Q^d &= L^2(\Omega^d), & 0 &\leq d \leq n.
	\end{align*}
	The key tool used to create a succinct method, is to create dimensionally structured function spaces by applying the direct sum over all different dimensionalities. Particularly, we define the composite function spaces
	\begin{align}
		\mathscr{V} &= \bigoplus_{d=1}^n \bm{V}^d, & 
		\mathit{\Lambda} &= \bigoplus_{d=0}^{n-1} \Lambda^d, & 
		\mathscr{Q} &= \bigoplus_{d=0}^n Q^d. \label{eq: direct_sum}
	\end{align}

	The dimensionally structured space $\mathit{\Lambda}$ will contain the normal flux across $\Gamma$ and act as a mortar space. To avoid doubly defining the normal fluxes across $\Gamma$ with functions from $\mathscr{V}$ and $\mathit{\Lambda}$, a final function space is defined containing functions with zero normal flux across $\Gamma$:
	\begin{align}
		\bm{V}^d_0 &= \left\{\bm{v} \in \bm{V}^d: \epsilon \bm{v} \cdot \bm{\nu} = 0 \text{ on } \Gamma^{d - 1} \right\}, & 1 &\leq d \leq n, \nonumber\\
		\mathscr{V}_0 &= \bigoplus_{d=1}^n \bm{V}^d_0.
	\end{align}

	To rigorously impose the essential boundary condition on $\Gamma$, a linear extension operator $\mathcal{R}$ is introduced for functions belonging to $\mathit{\Lambda}$. The construction of this operator is done using the dimensional decomposition. For $0 \le d \le n - 1$, let the operator $\mathcal{R}^d: \Lambda^d \to \bm{V}^{d + 1}$ be defined such that
	\begin{align}
		\mathcal{R}^d \lambda^d \cdot \bm{\nu} &= \begin{cases}
		\lambda^d \ \text{on} \ \Gamma^d \\ 0 \ \text{on} \ \partial \Omega^{d+1} \backslash \Gamma^d,
		\end{cases} \label{eq: ExtensionR}
	\end{align}
	in which $\bm{\nu}$ represents the unit normal vector associated with $\Gamma^d$. The image of $\mathcal{R}^d$ has slightly higher regularity than $H(\operatorname{div}; \Omega^{d+1})$ with normal trace in $L^2(\partial \Omega^{d+1})$. Now, let us define the operator $\mathcal{R}: \mathit{\Lambda} \to \mathscr{V}$ as
	\begin{align*}
		\mathcal{R} \lambda = \bigoplus_{d = 0}^{n-1} \mathcal{R}^d \lambda^d.
	\end{align*} 
	At this point, some freedom remains in the choice of $\mathcal{R}$. Even though the resulting method is not affected by the eventual choice, a specific extension operator is constructed later in \eqref{Rtilde} which has favorable properties for the sake of the analysis.

	Due to this construction, the flux will always be composed of a pair $(\bm{u}_0, \lambda)$ which gives rise to the space $\mathscr{X}$ given by
	\begin{align}
		\mathscr{X} = \mathscr{V}_0 \times \mathit{\Lambda}. \label{space X}
	\end{align}

	With the appropriate function spaces and operators defined, we continue with the derivation of the weak form of the problem. The derivation is standard for all equations except for \eqref{pressure jump}, which requires some additional attention. For a given $\Omega_i^d$, $0 \le d \le n-1$, let us test \eqref{pressure jump} with a function $\mu^d \in \Lambda^d$.	After summation over $j \in \mathcal{J}_i^d$, we obtain
	\begin{align}
		\sum_{j \in \mathcal{J}_i^d} \langle \frac{\gamma}{K_{\nu}} \lambda_{ij}^d, \mu_{ij}^d \rangle_{\Gamma_{ij}^d} &= \sum_{j \in \mathcal{J}_i^d} \langle p^{d + 1}, \hat{\epsilon} \mu^d_{ij} \rangle_{\Gamma_{ij}^d} + ( p_i^d, \llbracket \hat{\epsilon} \mu^d \rrbracket )_{\Omega_i^d}, \label{normalDarcy}
	\end{align}
	where $\langle \cdot, \cdot \rangle_{\Gamma_{ij}^d}$ and $(\cdot, \cdot)_{\Omega_i^d}$ denote the $L^2$-inner products on $\Gamma_{ij}^d$ and $\Omega_i^d$, respectively.
	A useful aspect of this relationship is that the first term on the right-hand side is exactly the boundary term which appears in the weak form of Darcy's law \eqref{Darcylower} after partial integration. The notation is simplified by introducing the inner products and the associated norms in the dimensional decomposition framework:
	\begin{align*}
		(\cdot,\cdot)_{\Omega} &= \sum_{d = 0}^{n} (\cdot,\cdot)_{\Omega^d} = \sum_{d = 0}^{n} \sum_{i = 1}^{N^d} (\cdot,\cdot)_{\Omega_i^d}, &
		\| \cdot \|_{L^2(\Omega)}^2 &= \sum_{d = 0}^{n} \| \cdot \|_{L^2(\Omega^d)}^2, \\
		\langle \cdot,\cdot \rangle_{\Gamma} &= \sum_{d = 0}^{n-1} \langle\cdot,\cdot\rangle_{\Gamma^d} = \sum_{d = 0}^{n-1} \sum_{i = 1}^{N^d} \sum_{j \in \mathcal{J}_i^d} \langle\cdot,\cdot\rangle_{\Gamma_{ij}^d} , &
		\| \cdot \|_{L^2(\Gamma)}^2 &= \sum_{d = 0}^{n-1} \| \cdot \|_{L^2(\Gamma^d)}^2.
	\end{align*}
	We are now ready to state the variational form of the problem:

	The weak solution $(\bm{u}_0, \lambda, p) \in \mathscr{V}_{0} \times \mathit{\Lambda} \times \mathscr{Q}$ satisfies
	\begin{subequations} \label{full system}
		\begin{align}
			(K^{-1} (\bm{u}_0 + \mathcal{R} \lambda), \bm{v}_0)_{\Omega} - (p, \nabla \cdot \epsilon \bm{v}_0)_{\Omega} 
			&= - \langle g, \epsilon \bm{v}_0 \cdot \bm{\nu} \rangle_{\partial \Omega_D}
			& \forall \bm{v}_0 &\in \mathscr{V}_0, \label{full u}\\
			(K^{-1} (\bm{u}_0 + \mathcal{R} \lambda), \mathcal{R} \mu)_{\Omega} 
			- ( p, \nabla \cdot \epsilon \mathcal{R} \mu)_{\Omega} \nonumber \\
			+ \langle \frac{\gamma}{K_{\nu}} \lambda, \mu \rangle_{\Gamma}
			- ( p, \llbracket \hat{\epsilon} \mu \rrbracket )_{\Omega} &= 0 & \forall \mu &\in \mathit{\Lambda}, \label{full lambda} \\
			-(\nabla \cdot \epsilon (\bm{u}_0 + \mathcal{R} \lambda), q)_{\Omega} 
			- ( \llbracket \hat{\epsilon} \lambda \rrbracket, q )_{\Omega}
			&= -(\epsilon^2 f,q)_{\Omega} 
			& \forall q &\in \mathscr{Q}. \label{full p}
		\end{align}
	\end{subequations}
	We set all functions not defined for certain indexes (such as $\bm{u}_0^0$ and $\lambda^n$) to zero such that the unified presentation is well-defined. Equation \eqref{full u} follows from \eqref{eq: Darcy} and \eqref{Darcylower}, whereas equation \eqref{full lambda} follows additionally from \eqref{normalDarcy}. Finally, equation \eqref{full p} follows from equations \eqref{eq: massconv}, \eqref{massconvlower} and \eqref{Massconvlower 0}. In the above, we assume that $g \in H^{\frac{1}{2}}(\partial \Omega_D)$ and $f \in L^2(\Omega)$ which guarantees that the right-hand side terms in \eqref{full u} and \eqref{full lambda} are well-posed. In particular, since $\epsilon \bm{v_0}^d \in H(\operatorname{div}; \Omega^d)$ and $\epsilon \bm{v_0}^d \cdot \bm{\nu} = 0$ on $\partial \Omega^d \backslash \partial \Omega_D^d$, then $\epsilon \bm{v_0}^d \cdot \bm{\nu} \in H^{-\frac{1}{2}}(\partial \Omega_D^d)$, see e.g. \cite{galvis2007non}.

	We note that for fractures which have $\epsilon = \gamma = 0$ uniformly, this model reduces to a domain decomposition method which uses $\lambda$ as a flux mortar to impose continuity of pressure in a weak sense. 
	
	The next step is to observe that the system \eqref{full system} can be classified as a saddle point problem. For this purpose, we rewrite the problem into a different format by using the divergence operator $D$ from \eqref{div D} and the bilinear forms $a$ and $b$ given by
	\begin{subequations} \label{operators}
		\begin{align}
			a(\bm{u}_0, \lambda;\ \bm{v}_0, \mu) &= 
			(K^{-1} (\bm{u}_0 + \mathcal{R} \lambda), \bm{v}_0 + \mathcal{R} \mu)_{\Omega}
			+ \langle \frac{\gamma}{K_{\nu}} \lambda, \mu \rangle_{\Gamma}, \label{operatora}\\
			b(\bm{v}_0, \mu;\ p) &= 
			- ( p, D \cdot [\epsilon (\bm{v}_0 + \mathcal{R} \mu), \hat{\epsilon} \mu])_{\Omega}. \label{operatorb}
		\end{align}
	\end{subequations}
		
	These definitions allows us to rewrite system \eqref{full system} to the following, equivalent problem: \\
	Find the functions ${(\bm{u}_0, \lambda, p) \in \mathscr{V}_{0} \times \mathit{\Lambda} \times \mathscr{Q}}$ such that
	\begin{align}
		a(\bm{u}_0, \lambda;\ \bm{v}_0, \mu) + 
				b(\bm{v}_0, \mu;\ p) 
				- b(\bm{u}_0, \lambda;\ q) &= -\langle g, \epsilon \bm{v}_0 \cdot \bm{\nu} \rangle_{\partial \Omega_D} + (\epsilon^2 f,q)_{\Omega}, \label{eq: weakform}
	\end{align}
	for all $(\bm{v}_0, \mu, q) \in \mathscr{V}_{0} \times \mathit{\Lambda} \times \mathscr{Q}$.

\subsection{Well-posedness}
\label{sec:well_posedness}

	Before proceeding to the discretization, it is important to analyze the variational problem \eqref{eq: weakform} in the continuous sense. To that end, we present a proof of the well-posedness of this problem within the dimensional hierarchy setting.
	
	For the purpose of the analysis, let us introduce a specific extension operator $\mathcal{R}: \mathit{\Lambda} \to \mathscr{V}$. For $1 \le d \le n$, let $\mathcal{R}^{d-1} \lambda^{d-1} \in \bm{V}^d$ and an auxiliary variable $p_{\lambda}^d \in Q^d$ be defined as the solution to the following problem:
		\begin{subequations} \label{Rtilde}
			\begin{align}
				(K^{-1} \mathcal{R}^{d-1} \lambda^{d-1}, \bm{v}_0^d)_{\Omega^d} - 
				(p_\lambda^d, 
				\nabla \cdot \epsilon\bm{v}_0^d)_{\Omega^d} &= 0 & \forall \bm{v}_0^d &\in \bm{V}_0^d, \\
				(\nabla \cdot \epsilon\mathcal{R}^{d-1} \lambda^{d-1}, q^d)_{\Omega^d} + (\epsilon p_\lambda^d, q^d)_{\Omega^d} &= 0 & \forall q^d &\in Q^d, \label{c} \\
				\mathcal{R}^{d-1} \lambda^{d-1} \cdot \bm{\nu} &= \lambda^{d-1} &&\text{on } \Gamma^{d-1}, \\
				\mathcal{R}^{d-1} \lambda^{d-1} \cdot \bm{\nu} &= 0 &&\text{on } \partial \Omega^d \backslash \Gamma^{d-1}.
			\end{align}
		\end{subequations}
	Note that the boundary conditions are chosen such that $\mathcal{R}^{d-1} \lambda^{d-1}$ is a suitable extension compliant with equation \eqref{eq: ExtensionR}. 

	\begin{lemma} \label{EllipticR}
		The solution $(\mathcal{R}^{d-1} \lambda^{d-1}, p_{\lambda}^d) \in \bm{V}^d \times Q^d$ to problem \eqref{Rtilde} satisfies the following bounds:
		\begin{subequations}
			\begin{align}
				\| K^{-\frac{1}{2}} \mathcal{R}^{d-1} \lambda^{d-1} \|_{L^2(\Omega^d)} + \| \epsilon^{\frac{1}{2}} p_{\lambda}^d \|_{L^2(\Omega^d)} &\lesssim \| \lambda^{d-1} \|_{L^2(\Gamma^{d-1})}, \label{bound Rlambda} \\
				\| \nabla \cdot \epsilon \mathcal{R}^{d-1} \lambda^{d-1} \|_{L^2(\Omega^d)} &\lesssim \| \epsilon_{\max}^{\frac{1}{2}} \lambda^{d-1} \|_{L^2(\Gamma^{d-1})}, \label{bound div Rlambda}
			\end{align}	
		\end{subequations}
		where $\epsilon_{\max}|_{\Omega_i^d} = \| \epsilon \|_{L^{\infty}(\Omega_i^d)}$.
	\end{lemma}
	\begin{proof}
		Let us introduce the function $\bm{v}_\lambda^d$ as the $H(\operatorname{div})$-extension of $\lambda^{d-1}$ described in \cite{quarteroni1999domain} (Section 4.1.2). In particular, $\bm{v}_\lambda^d \cdot {\bf{\nu}} = \lambda^{d-1}$ and it satisfies the following bound:
		\begin{align}
			\| \bm{v}_\lambda^d \|_{L^2(\Omega^d)}^2 + \| \nabla \cdot \bm{v}_\lambda^d \|_{L^2(\Omega^d)}^2
			&\lesssim \| \lambda^{d-1} \|_{L^2(\Gamma^{d-1})}^2. \label{hdiv extension}
		\end{align}

		Inequality \eqref{bound Rlambda} is formed by setting the test functions in \eqref{Rtilde} as $\bm{v}_0^d = \mathcal{R}^{d-1} \lambda^{d-1} - \bm{v}_\lambda^d$ and $q^d = p_{\lambda}^d$. After summation of the equations, we obtain
		\begin{align*}			
			\| K^{-\frac{1}{2}} \mathcal{R}^{d-1} \lambda^{d-1} \|_{L^2(\Omega^d)}^2 &+ \| \epsilon^{\frac{1}{2}} p_{\lambda}^d \|_{L^2(\Omega^d)}^2 =
			(K^{-1} \mathcal{R}^{d-1} \lambda^{d-1}, \bm{v}_\lambda^d)_{\Omega^d} + 
				(p_{\lambda}^d, 
				\nabla \cdot \epsilon\bm{v}_\lambda^d)_{\Omega^d} \nonumber \\
			&=
			(K^{-1} \mathcal{R}^{d-1} \lambda^{d-1}, \bm{v}_\lambda^d)_{\Omega^d} + (p \nabla \epsilon, \bm{v}_\lambda^d)_{\Omega^d} + 
				(\epsilon p_\lambda^d, 
				\nabla \cdot \bm{v}_\lambda^d)_{\Omega^d}.
		\end{align*}
		The Cauchy-Schwarz inequality is then used followed by the positive-definiteness of $K$, the bound on $\nabla \epsilon$ from \eqref{grad eps bound}, and \eqref{hdiv extension} to give:
		\begin{align}			
			\| K^{-\frac{1}{2}} \mathcal{R}^{d-1} \lambda^{d-1} \|_{L^2(\Omega^d)}^2 
			&+ \| \epsilon^{\frac{1}{2}} p_{\lambda}^d \|_{L^2(\Omega^d)}^2 \nonumber \\
			&\lesssim
			\left(\| K^{-\frac{1}{2}} \mathcal{R}^{d-1} \lambda^{d-1} \|_{L^2(\Omega^d)} 
			+ \| \epsilon^{\frac{1}{2}} p_{\lambda}^d \|_{L^2(\Omega^d)} \right)
			\| \lambda^{d-1} \|_{L^2(\Gamma^{d-1})}. \label{first bound cont}
		\end{align}

		Secondly, we obtain \eqref{bound div Rlambda} by setting $q^d = \nabla \cdot \epsilon \mathcal{R}^{d-1} \lambda^{d-1}$ in equation \eqref{c}:
		\begin{align*}
			\| \nabla \cdot \epsilon \mathcal{R}^{d-1} \lambda^{d-1} \|_{L^2(\Omega^d)} &\le
			\| \epsilon p_{\lambda}^d \|_{L^2(\Omega^d)} 
			\le \| \epsilon_{\max}^{\frac{1}{2}} \epsilon^{\frac{1}{2}} p_{\lambda}^d \|_{L^2(\Omega^d)}
			\lesssim \| \epsilon_{\max}^{\frac{1}{2}} \lambda^{d-1} \|_{L^2(\Gamma^{d-1})}.
		\end{align*}
	\end{proof}

	The constructed extension operator $\mathcal{R}$ allows us to form the norms as used in the subsequent analysis:
	\begin{subequations} \label{norms cont}
		\begin{align}
			\| [\bm{v}_0, \mu ] \|_{\mathscr{X}_{\mathcal{R}}}^2 =&\  
			\| K^{-\frac{1}{2}} (\bm{v}_0 + \mathcal{R} \mu)\|_{L^2(\Omega)}^2 
			+ \| \gamma^{\frac{1}{2}} K_{\nu}^{-\frac{1}{2}} \mu \|_{L^2(\Gamma)}^2 \nonumber \\
			&+ \| \hat{\epsilon}_{\max}^{-1} D \cdot [\epsilon (\bm{v}_0 + \mathcal{R} \mu), \hat{\epsilon} \mu ] \|_{L^2(\Omega)}^2, \label{X-norm cont}\\
			\| q \|_{\mathscr{Q}} =& \ \| \hat{\epsilon}_{\max} q \|_{L^2(\Omega)}. \label{p-norm cont}
		\end{align}
	\end{subequations}
	Here, $\hat{\epsilon}_{\max}$ is used as defined in \eqref{eps_max bound}. The energy norm is created as the combination of these norms:
	\begin{align}
		\enorm{(\bm{u}_0, \lambda, p)}^2 &= \| [\bm{u}_0, \lambda ] \|_{\mathscr{X}_{\mathcal{R}}}^2 + \| p \|^2_{\mathscr{Q}}. \label{eq: enorm}
	\end{align}

	In order to show well-posedness of the problem in this energy norm, we present three lemmas, which provide the necessary tools to invoke standard saddle-point theory.

	\begin{lemma}[Completeness] \label{Lem: Completeness}
		With the extension operator $\mathcal{R}$ from \eqref{Rtilde}, the space $\mathscr{X}$ from \eqref{space X} is a Hilbert space with inner product
		\begin{align}
			( [\bm{u}_0, \lambda ], [\bm{v}_0, \mu ] )_{\mathscr{X}_{\mathcal{R}}} =&\  
			( K^{-1} (\bm{u}_0 + \mathcal{R} \lambda), \bm{v}_0 + \mathcal{R} \mu)_{L^2(\Omega)} 
			+ ( \gamma K_{\nu}^{-1} \lambda, \mu )_{L^2(\Gamma)} \nonumber \\
			&+ ( \hat{\epsilon}_{\max}^{-1} D \cdot [\epsilon (\bm{u}_0 + \mathcal{R} \lambda), \hat{\epsilon} \lambda ], \hat{\epsilon}_{\max}^{-1} D \cdot [\epsilon (\bm{v}_0 + \mathcal{R} \mu), \hat{\epsilon} \mu ] )_{L^2(\Omega)},
		\end{align}
		which induces the norm from \eqref{X-norm cont}.
	\end{lemma}
	\begin{proof}
		$\mathscr{X}$ is a linear space and $( \cdot, \cdot)_{\mathscr{X}_{\mathcal{R}}}$ is an inner product. In order to show completeness of $\mathscr{X}$ with respect to the induced norm \eqref{X-norm cont}, we consider a Cauchy sequence $\{ [\bm{v}_{0,k}, \mu_k ] \}_{k = 0}^\infty \subset \mathscr{X}$. In other words, as $l, k \to \infty$, we have
		\begin{align}
			\| [\bm{v}_{0,k} - \bm{v}_{0,l}, \mu_k - \mu_l ] \|_{\mathscr{X}_{\mathcal{R}}}^2 
			\to&\ 0.
		\end{align}

		By completeness of the $L^2$-spaces, there exists a $\bm{v} \in L^2(\Omega)$ such that $\bm{v}_{0, k} + \mathcal{R} \mu_k \to \bm{v}$ and a $\mu \in L^2(\Gamma)$ such that $\mu_k \to \mu$, using \eqref{gamma_0} for the latter. 
		Thus, we can define $\bm{v}_0 = \bm{v} - \mathcal{R} \mu \in L^2(\Omega)$. Using the same argumentation, $\xi \in L^2(\Omega)$ exists such that $\hat{\epsilon}_{\max}^{-1} D \cdot [\epsilon (\bm{v}_{0,k} + \mathcal{R} \mu_k), \hat{\epsilon} \mu_k ] \to \xi$. It remains to show how $\xi$ is connected to $[\bm{v}_0, \mu]$.

		Let us consider a test function $\psi$ with $\psi^0 \in L^2(\Omega^0)$ and $\psi^d \in C_0^\infty(\Omega^d)$ for $d \ge 1$ and derive
		\begin{align}
			(\hat{\epsilon}_{\max}^{-1} D \cdot [\epsilon (\bm{v}_{0,k} &+ \mathcal{R} \mu_k), \hat{\epsilon} \mu_k ], \psi)_{\Omega} = 
			(\hat{\epsilon}_{\max}^{-1} \nabla \cdot \epsilon (\bm{v}_{0,k} + \mathcal{R} \mu_k), \psi)_{\Omega} 
			+ (\hat{\epsilon}_{\max}^{-1} \llbracket \hat{\epsilon} \mu_k \rrbracket, \psi)_{\Omega} \nonumber \\
			=&\ - (\epsilon (\bm{v}_{0,k} + \mathcal{R} \mu_k), \nabla \hat{\epsilon}_{\max}^{-1} \psi)_{\Omega}
			+ (\hat{\epsilon}_{\max}^{-1} \llbracket \hat{\epsilon} \mu_k \rrbracket, \psi)_{\Omega} \nonumber \\
			=&\ -(\bm{v}_{0,k} + \mathcal{R} \mu_k, -\epsilon \hat{\epsilon}_{\max}^{-2} (\nabla \hat{\epsilon}_{\max})\psi + \epsilon \hat{\epsilon}_{\max}^{-1} (\nabla \psi))_{\Omega} 
			+ (\hat{\epsilon}_{\max}^{-1} \llbracket \hat{\epsilon} \mu_k \rrbracket, \psi)_{\Omega} \nonumber \\
			\xrightarrow{k\to \infty}&\ -(\bm{v}_0 + \mathcal{R}\mu, -\epsilon \hat{\epsilon}_{\max}^{-2} (\nabla \hat{\epsilon}_{\max})\psi + \epsilon \hat{\epsilon}_{\max}^{-1} (\nabla \psi))_{\Omega} 
			+ (\hat{\epsilon}_{\max}^{-1} \llbracket \hat{\epsilon} \mu \rrbracket, \psi)_{\Omega} \nonumber \\
			=&\ (\hat{\epsilon}_{\max}^{-1} D \cdot [\epsilon \bm{v}_0, \hat{\epsilon} \mu ], \psi)_{\Omega}
		\end{align}
		Hence, we have shown that $\xi = \hat{\epsilon}_{\max}^{-1} D \cdot [\epsilon \bm{v}_0, \hat{\epsilon} \mu ]$. Moreover, since $\mu \in L^2(\Gamma)$, it follows that $\llbracket \mu \rrbracket \in L^2(\Omega)$. With $\xi \in L^2(\Omega)$, we obtain $\nabla \cdot \epsilon \bm{v}_0 \in L^2(\Omega)$ and therewith $\bm{v}_0 \in \mathscr{V}_0$. Thus, $\mathscr{X}$ is complete.
	\end{proof}

	\begin{remark}
	The above proof exploits the lower bound on $\gamma K_{\nu}^{-1}$ stated in \eqref{gamma_0}. In order to avoid this restriction, weighted Sobolev spaces need to be considered similar to e.g. \cite{girault2015lubrication}.\end{remark}

	\begin{lemma}[Continuity and Ellipticity] \label{ElKerCont}
	The bilinear forms $a$ and $b$ from \eqref{operators} are continuous with respect to the norms given in \eqref{norms cont}. Moreover, if the pair $(\bm{u}_0, \lambda) \in \mathscr{X}$ satisfies
		\begin{align}
			b(\bm{u}_0, \lambda;\ q) = 0 \text{ for all } q \in \mathscr{Q}, \label{assumption on b}
		\end{align}
		then a constant $C_a > 0$ exists such that
		\begin{align}
			a(\bm{u}_0, \lambda;\ \bm{u}_0, \lambda) \geq C_a \enorm{(\bm{u}_0, \lambda, 0)}^2.
		\end{align}
	\end{lemma}
	\begin{proof}
		Continuity of the bilinear forms follows directly from the Cauchy-Schwarz inequality. Let us continue with assumption \eqref{assumption on b}, which translates to
		\begin{align*}
			( q, D \cdot [\epsilon (\bm{u}_0 + \mathcal{R} \lambda), \hat{\epsilon} \lambda ] )_{\Omega} &= 0,& \text{for all }q &\in \mathscr{Q}.
		\end{align*}
		Since $D \cdot \mathscr{X} \subseteq \mathscr{Q}$, it follows that 
		\begin{align*}
		 	\| D \cdot [\epsilon (\bm{u}_0 + \mathcal{R} \lambda), \hat{\epsilon} \lambda ] \|_{L^2(\Omega)}^2 = 0.
		\end{align*}
		Using the definition of $a$ from \eqref{operatora} and $\hat{\epsilon}_{\max} > 0$ from \eqref{eps_max bound}, we obtain:
		\begin{align*}
			a(\bm{u}_0, \lambda&;\ \bm{u}_0, \lambda) = 
			(K^{-1} (\bm{u}_0 + \mathcal{R} \lambda), \bm{u}_0 + \mathcal{R} \lambda)_{\Omega}
			+ \langle \frac{\gamma}{K_{\nu}} \lambda, \lambda \rangle_{\Gamma} \nonumber\\
			&= \| K^{-\frac{1}{2}} (\bm{u}_0 + \mathcal{R} \lambda) \|_{L^2(\Omega)}^2 
			+ \| \gamma^{\frac{1}{2}} K_{\nu}^{-1} \lambda \|_{L^2(\Gamma)}^2 
			+ \| \hat{\epsilon}_{\max}^{-1} D \cdot [\epsilon (\bm{u}_0 + \mathcal{R} \lambda), \hat{\epsilon} \lambda ] \|_{L^2(\Omega)}^2 \\
			&= \enorm{(\bm{u}_0, \lambda, 0)}^2.
		\end{align*}
		Thus, the result is shown with $C_a = 1$.
	\end{proof}

	\begin{lemma}[Inf-Sup] \label{lem: inf-sup b cont}
		Let the bilinear form $b$ be defined by equation \eqref{operatorb}. Then there exists a constant $C_b>0$ such that for any given function $p \in \mathscr{Q}$, 
		\begin{align}
			\sup_{(\bm{v}_0, \mu) \in \mathscr{X}} \frac{b(\bm{v}_0, \mu;\ p)}{\enorm{(\bm{v}_0, \mu, 0)} } &\geq C_b \enorm{(0,0, p)}. \label{eq: infsup1} 
		\end{align}
	\end{lemma}
	\begin{proof} 
		Assume $p \in \mathscr{Q}$ given. We aim to construct a pair $(\bm{v}_0, \mu) \in \mathscr{X}$ such that the inequality holds. The construction is done by sequentially ascending through the dimensional hierarchy. For convenience, we recall the definition of $b$:
		\begin{align*}	
			-b(\bm{v}_0, \mu;\ p) &=
			( p, \nabla \cdot \epsilon(\bm{v}_0 + \mathcal{R} \mu))_{\Omega}
			+ ( p, \llbracket \hat{\epsilon} \mu \rrbracket )_{\Omega}.
		\end{align*}

		The function $\mu \in \mathit{\Lambda}$ is constructed in a hierarchical manner. Let us start by choosing $\mu^0$ such that the following is satisfied:
		\begin{align}
			\llbracket \hat{\epsilon} \mu^0 \rrbracket &= \hat{\epsilon}_{\max}^2 p^0 , 
			& \| \mu^0 \|_{L^2(\Gamma^0)} \lesssim \| \hat{\epsilon}_{\max} p^0 \|_{L^2(\Omega^0)}. \label{boundmu}
		\end{align}
		We construct a suitable $\mu^0$ for a given index $i$ by finding $j_{\max}$ where $\hat{\epsilon}_{ij_{\max}} = \hat{\epsilon}_{\max}$ and setting $\mu_{ij_{\max}}^0 = \hat{\epsilon}_{\max} p^0$ while choosing $\mu_{ik}^0 = 0$ for $k \ne j_{\max}$.
		
		The next step is to generalize this strategy to $1 \le d \le n-1$. In this, we need to counteract the contribution of the extension operator. Let us construct $\mu^d$ such that it satisfies:
		\begin{align}
			\llbracket \hat{\epsilon} \mu^d \rrbracket &= 
			\hat{\epsilon}_{\max}^2 p^d 
			- \nabla \cdot \epsilon \mathcal{R}^{d-1} \mu^{d-1}.
		\end{align}
		Again, only $\mu_{ij_{\max}}^d$ is non-zero. We now have
		\begin{align}
			\| \mu^d \|_{L^2(\Gamma^d)} &\lesssim 
			\| \hat{\epsilon}_{\max} p^d \|_{L^2(\Omega^d)}
			+ \| \hat{\epsilon}_{\max}^{-1} \nabla \cdot \epsilon \mathcal{R}^{d-1} \mu^{d-1} \|_{L^2(\Omega^d)} \nonumber\\
			&\lesssim
			\| \hat{\epsilon}_{\max} p^d \|_{L^2(\Omega^d)}
			+ \| \hat{\epsilon}_{\max}^{-1} \epsilon_{\max}^{\frac{1}{2}} \mu^{d-1} \|_{L^2(\Omega^d)} \nonumber\\
			&\lesssim \| \hat{\epsilon}_{\max} p^d \|_{L^2(\Omega^d)}
			+ \| \mu^{d-1} \|_{L^2(\Gamma^{d-1})}, \label{boundmu2}
		\end{align}
		where we used Lemma~\ref{EllipticR} and property \eqref{eps upperbounded by eps_max}. 

		Next, we set the functions $\bm{v}_0^d$ with $1 \le d \le n-1$ to zero and continue with $d = n$. Let us construct $\bm{v}_0^n \in \bm{V}_0^n$ and a supplementary variable $p_v^n \in Q^n$ using the following auxiliary problem :
		\begin{align*}
			(K^{-1}  \bm{v}_0^n, \bm{w}_0^n)_{\Omega^n} - 
			(p_v^n, 
			\nabla \cdot \bm{w}_0^n)_{\Omega^n} &= 0 & \forall \bm{w}_0^n &\in \bm{V}_0^n, \\
			(\nabla \cdot  \bm{v}_0^n, q^n)_{\Omega^n} &= (p^n - \nabla \cdot \mathcal{R}^{n-1} \mu^{n-1}, q^n)_{\Omega^n}  & \forall q^n &\in Q^n.
		\end{align*}
		This problem is well-posed since $|\partial \Omega_i^n \cap \partial \Omega_D| > 0$ for each $i$ and thus each subdomain borders on a homogeneous, Dirichlet boundary condition. 
		Standard stability arguments for this mixed formulation combined with the estimate from Lemma~\ref{EllipticR} and the defined $\epsilon = 1$ in $\Omega^n$ then give us
		\begin{align}
			\| K^{-\frac{1}{2}} \bm{v}_0^n \|_{L^2(\Omega^n)}^2 + \| \nabla \cdot \epsilon \bm{v}_0^n \|_{L^2(\Omega^n)}^2
			&\lesssim \| p^n \|_{L^2(\Omega^n)} + \| \nabla \cdot \mathcal{R}^{n-1} \mu^{n-1} \|_{L^2(\Omega^n)} \nonumber\\
			&\lesssim \| p^n \|_{L^2(\Omega^n)} + \| \mu^{n-1} \|_{L^2(\Gamma^{n-1})}. \label{ellipticphi}
		\end{align}

		The choice $(\bm{v}_0, \mu) \in \mathscr{V}_0 \times \mathit{\Lambda}$ is now finalized and two key observations can be made. First, we recall the positive-definiteness of $K_\nu$ and the boundedness of $\mathcal{R}$ given by Lemma~\ref{EllipticR}. Combined with the bounds \eqref{boundmu}, \eqref{boundmu2}, and \eqref{ellipticphi}, we derive using \eqref{eps upperbounded by eps_max}:
		\begin{align}
			\| [\bm{v}_0, \mu ] \|_{\mathscr{X}_{\mathcal{R}}}^2
			\lesssim&\ \| K^{-\frac{1}{2}} (\bm{v}_0 + \mathcal{R} \mu) \|_{L^2(\Omega)}^2 
			+ \| \mu \|_{L^2(\Gamma)}^2 
			+ \| \hat{\epsilon}_{\max}^{-1} D \cdot [\epsilon (\bm{v}_0 + \mathcal{R} \mu), \hat{\epsilon} \mu ] \|_{L^2(\Omega)}^2 \nonumber\\
			\lesssim&\ \| p^n \|_{L^2(\Omega^n)}
			+ \| \mu \|_{L^2(\Gamma)}^2 
			+ \| K^{-\frac{1}{2}} \mathcal{R} \mu\|_{L^2(\Omega)}^2  \nonumber\\
			&+ \| \hat{\epsilon}_{\max}^{-1} \nabla \cdot \epsilon \mathcal{R} \mu\|_{L^2(\Omega)}^2
			+ \| \hat{\epsilon}_{\max}^{-1} \llbracket \hat{\epsilon} \mu \rrbracket \|_{L^2(\Omega)}^2 \nonumber\\
			\lesssim&\ \| p^n \|_{L^2(\Omega^n)}^2
			+ \| \mu \|_{L^2(\Gamma)}^2 \nonumber\\
			\lesssim&\ \| p^n \|_{L^2(\Omega^n)}^2
			+ \textstyle{\sum_{d=0}^{n-1}} \| \hat{\epsilon}_{\max} p^d \|_{L^2(\Omega^d)}^2
			= \enorm{(0,0, p)}^2. \label{result2}
		\end{align}

		Moreover, substitution of the constructed $(\bm{v}_0, \mu)$ in the form $b$ gives us
		\begin{align*}
			(p^d, \nabla \cdot \epsilon (\bm{v}^d_0 + \mathcal{R}^{d-1} \mu^{d-1}))_{\Omega^d}
			+ ( p^d, \llbracket \hat{\epsilon} \mu^d \rrbracket )_{\Omega^d}
			&= \| \hat{\epsilon}_{\max} p^d \|_{L^2(\Omega^d)}^2, & 0 \le d \le n.
		\end{align*}
		Thus, after summation over all dimensions, we obtain
		\begin{align}
			b(\bm{v}_0, \mu;\ p) = \enorm{(0,0, p)}^2. \label{result1}
		\end{align}
		The proof is concluded by combining \eqref{result2} and \eqref{result1}.
	\end{proof}

	We emphasize that the constants used in the previous lemmas are independent of $\gamma$ and $\epsilon$. In fact, the dependency on the aperture is completely reflected in the definition of the norms. 
	\begin{theorem} \label{well-posedness cont}
		Problem \eqref{eq: weakform} is well-posed with respect to the energy norm \eqref{eq: enorm}, i.e. there exists a unique solution such that
		\begin{align}
			\enorm{(\bm{u}_0, \lambda, p)} \lesssim \| \epsilon^{\frac{3}{2}} f \|_{L^2(\Omega)} + \| g \|_{H^{\frac{1}{2}}(\partial \Omega_D)}. 
		\end{align}
	\end{theorem}
	\begin{proof}
		We firstshow the continuity of the right-hand side of \eqref{eq: weakform}. We consider each term separately:
		\begin{align}
			-\langle g, \epsilon \bm{v}_0 \cdot \bm{\nu} \rangle_{\partial \Omega_D} 
			=&\ -\langle g, \epsilon (\bm{v}_0 + \mathcal{R} \mu) \cdot \bm{\nu} \rangle_{\partial \Omega_D} \nonumber\\
			\lesssim&\ \| g \|_{H^{\frac{1}{2}}(\partial \Omega_D)}
			\| \epsilon (\bm{v}_0 + \mathcal{R} \mu) \|_{H(\operatorname{div}, \Omega)} \nonumber\\
			\lesssim&\ \| g \|_{H^{\frac{1}{2}}(\partial \Omega_D)}
			(\| \epsilon (\bm{v}_0 + \mathcal{R} \mu) \|_{L^2(\Omega)} +
			\| D \cdot [\epsilon (\bm{v}_0 + \mathcal{R} \mu), \hat{\epsilon} \mu ] \|_{L^2(\Omega)} \nonumber\\
			&+
			\| \llbracket \hat{\epsilon} \mu \rrbracket \|_{L^2(\Omega)}) \nonumber\\
			\lesssim&\ \| g \|_{H^{\frac{1}{2}}(\partial \Omega_D)}
			\| [\bm{v}_0, \mu ] \|_{\mathscr{X}_{\mathcal{R}}}.
		\end{align}
		Here, we used assumption \eqref{gamma_0} and $\hat{\epsilon}_{\max} > 0$ in the final step. For the second term, we use \eqref{eps upperbounded by eps_max} to derive:
		\begin{align}
			(\epsilon^2 f,q)_{\Omega}
			&\le \| \hat{\epsilon}_{\max}^{-1} \epsilon^2 f \|_{L^2(\Omega)} \| \hat{\epsilon}_{\max} q \|_{L^2(\Omega)}
			\le \| \epsilon^{\frac{3}{2}} f \|_{L^2(\Omega)} \| q \|_{\mathscr{Q}}.
		\end{align}
		Using Lemmas~\ref{Lem: Completeness} to \ref{lem: inf-sup b cont}, the well-posedness of problem \eqref{eq: weakform} follows from saddle-point problem theory, see e.g. \cite{Boffi}. 
	\end{proof}

\section{Discretization} 
\label{sec:discretization}

	In this section, the discretization of problem \eqref{eq: weakform} is considered. First, the requirements on the choice of discrete function spaces are stated. We then continue by showing stability for the discrete problem and end the section with a priori error estimates. 

\subsection{Discrete Spaces} 
\label{sub:discrete_spaces}

	In order to properly define the discrete equivalent of \eqref{eq: weakform}, we start by introducing the mesh. Let $\mathcal{T}^d_{\Omega}$ with $0\leq d \leq n$ be a finite element partition of $\Omega^d$ made up of $d$-dimensional, shape-regular, simplicial elements. Secondly, let $\mathcal{T}^d_{\Gamma}$ with $0\leq d \leq n - 1$ be a partition of $\Gamma^d$ consisting of $d$-dimensional simplices. We will commonly refer to $\mathcal{T}^d_{\Gamma}$ as the mortar mesh. Furthermore, let $h$ denote the mesh size.

	The discrete analogues of the function spaces are constructed using the dimensional hierarchy. Let us introduce $\bm{V}^d_h \subset \bm{V}^d$ and $\bm{V}^d_{0,h} \subset \bm{V}^d_0$ for $1 \le d \le n$ and $Q^d_h \subset Q^d$ with $0 \le d \le n$. Finally, the mortar space is given by
	\begin{align*}
		\Lambda^d_{ij, h} &\subset L^2(\Gamma_{ij}^d), & \Lambda^d_h &= \bigoplus_{i=1}^{N^d} \bigoplus_{j \in \mathcal{J}_i^d} \Lambda^d_{ij, h}, & 0 &\le d \le n-1.
	\end{align*}
	The discrete, dimensionally composite function spaces are then defined in analogy to \eqref{eq: direct_sum} as
	\begin{align*}
		\mathscr{V}_h &= \bigoplus_{d=1}^n \bm{V}_h^d, 
		& \mathscr{V}_{0,h} &= \bigoplus_{d=1}^n \bm{V}_{0,h}^d, 
		& \mathit{\Lambda}_h &= \bigoplus_{d=0}^{n-1} \Lambda_h^d, 
		& \mathscr{Q}_h &= \bigoplus_{d=0}^n Q_h^d.
	\end{align*}
	Finally, the combined space containing the fluxes is given by
	\begin{align*}
		\mathscr{X}_h = \mathscr{V}_{0, h} \times \mathit{\Lambda}_h.
	\end{align*}

	Before we continue with the analysis in Subsection~\ref{sub:a_priori_estimates}, let us present a total of four conditions on the discrete function spaces. The first is necessary, while the remaining conditions provide attractive features of the numerical method.

	First, it is essential that the pair $\mathscr{V}_{h} \times \mathscr{Q}_h$ is chosen such that
	\begin{align}
		Q^d_h &= \nabla \cdot V_{h}^d, & 1 &\le d \le n. \label{Fortin}
	\end{align}
	This can be satisfied by choosing any of the usual mixed finite element pairs \cite{Arbogast,Boffi}. 

	The second condition concerns the space $\mathit{\Lambda}_h$. For simplicity, we assume that the function spaces defined on different sides bordering $\Omega^d_i$ are the same. In other words, we have
	\begin{align*}
		\Lambda^d_{ij,h} &= \Lambda^d_{ik,h}, & j, k \in \mathcal{J}_i^d.
	\end{align*} 

	Third, conventional mortar methods (e.g. \cite{Arbogast}) require that the mortar mesh $\mathcal{T}^d_{\Gamma}$ is a sufficiently coarse partition of $\Gamma^d$ when compared to $\mathcal{T}^{d+1}_{\Omega}$. Let us define $\hat{\Pi}_h^d : \Lambda_h^d \to \bm{V}_h^{d+1} \cdot \bm{\nu}|_{\Gamma^d}$ as the $L^2$-projection from the mortar mesh onto the trace of the bordering, higher-dimensional mesh. In the unified setting, the projection $\hat{\Pi}_h$ is then given by $\bigoplus_{d = 0}^{n-1} \hat{\Pi}_h^d$ and the mortar condition can be described for $\mu_h \in \mathit{\Lambda}_h$ as
	\begin{align}
		\|\hat{\Pi}_h \mu_h \|_{L^2(\Gamma)} \gtrsim \|\mu_h \|_{L^2(\Gamma)}. \label{eq: CR}
	\end{align}
	This can easily be satisfied in case of matching grids by aligning the mortar grid with the trace of the surrounding mesh. Otherwise, it suffices to choose $\mathcal{T}^d_{\Gamma}$ as slightly coarser.

	As shown in \cite{Roberts}, the introduction of a flow problem inside the fracture guarantees a unique solution even if the mortar mesh is finer, thus removing the need for \eqref{eq: CR}. The same principle applies here. However, in this work we choose the mortar variable as the normal flux, instead of the fracture pressure, in order to have a stronger notion of mass conservation. Due to this choice, the control on the $L^2$-norm of the mortar variable is weighted with $\gamma$, as is apparent from \eqref{X-norm cont}. Since $\gamma$ is typically small, the main control on $\mu$ comes from $\mathcal{R}_h \mu$, which only sees $\hat{\Pi}_h \mu_h$ as boundary data. Thus, in order to eliminate the possible non-zero kernel of $\hat{\Pi}_h$, which may result in numerical oscillations of the mortar flux, it is advantageous to satisfy \eqref{eq: CR} in practice.

	Fourth, we let all lower-dimensional meshes match with the corresponding mortar mesh, such that
	\begin{align}
		\llbracket \mathit{\Lambda}_{h} \rrbracket &= \mathscr{Q}_h. \label{Q_h equals jump Lambda_h}
	\end{align}

	In the discretized setting, we have need of a discrete extension operator $\mathcal{R}_h: \mathit{\Lambda}_h \to \mathscr{V}_h$. In accordance with \eqref{eq: ExtensionR}, the function $\mathcal{R}_h \mu$ is such that $\mathcal{R}_h \mu \cdot \bm{\nu}|_{\Gamma} = \hat{\Pi}_h \mu$ and has zero normal trace on the remaining boundaries. A particularly attractive choice is to construct $\mathcal{R}_h \mu$ with a predefined support near the boundary. The bounded support then results in a beneficial sparsity pattern.

	To finish the section, we explicitly state a family of discrete function spaces which satisfy all conditions on simplicial elements for $n = 3$ and polynomial order $k$. Any choice of stable mixed spaces is valid and our choice is given by
		\begin{align}
	 		\label{Lowest order choice}
			\mathscr{V}_h &= 
			\textstyle{\bigoplus_{d = 1}^{3}} \mathbb{RT}_k(\mathcal{T}_{\Omega}^d), &
			\mathscr{Q}_h &= \textstyle{\bigoplus_{d = 0}^{3}} \mathbb{P}_k(\mathcal{T}_{\Omega}^d), & 
			\mathit{\Lambda}_{h} &= \textstyle{\bigoplus_{d = 0}^{2}} \mathbb{P}_{k}(\mathcal{T}_{\Gamma}^d).
		\end{align}
	Here, $\mathbb{RT}_k$ represents the $k$-th order Raviart-Thomas(-Nedelec) space \cite{Nedelec,Raviart-Thomas} which corresponds with continuous Lagrange elements of order $k+1$ for $d = 1$. The space $\mathbb{P}_k$ then represents $k$-th order discontinuous Lagrange elements. As is required, we choose $\mathscr{V}_{0,h} = \mathscr{V}_{0} \cap \mathscr{V}_{h}$ with zero normal trace on $\Gamma$. The function spaces corresponding to $k = 0$ will be referred to as the lowest-order choice.

	With the chosen discrete spaces, we are ready to define the discrete functionals. In the remainder, we will omit the index $h$ in most places for notational simplicity.
	\begin{subequations} \label{operators_h}
		\begin{align}
			a_h(\bm{u}_0, \lambda;\ \bm{v}_0, \mu) &= 
			(K^{-1} (\bm{u}_0 + \mathcal{R}_h \lambda), \bm{v}_0 + \mathcal{R}_h \mu)_{\Omega}
			+ \langle \frac{\gamma}{K_{\nu}} \lambda, \mu \rangle_{\Gamma}, \\
			b_h(\bm{v}_0, \mu;\ p) &= - ( p, D \cdot [\epsilon (\bm{v}_0 + \mathcal{R}_h \mu), \hat{\epsilon} \mu])_{\Omega}. \label{operatorb_h}
		\end{align}
	\end{subequations}

	The finite element problem associated with \eqref{eq: weakform} is now formulated as follows:
	Find $(\bm{u}_0, \lambda, p) \in \mathscr{V}_{0,h} \times \mathit{\Lambda}_h \times \mathscr{Q}_h$ such that 
	\begin{align}
		a_h(\bm{u}_0, \lambda;\ \bm{v}_0, \mu) + 
				b_h(\bm{v}_0, \mu;\ p) 
				- b_h(\bm{u}_0, \lambda;\ q) &= -\langle g, \epsilon \bm{v}_0 \cdot \bm{\nu} \rangle_{\partial \Omega_D} + (\epsilon^2 f,q)_{\Omega},  \label{weakform_h}
	\end{align}
	for all $(\bm{v}_0, \mu, q) \in \mathscr{V}_{0,h} \times \mathit{\Lambda}_h \times \mathscr{Q}_h$.

\subsection{Stability and Convergence} 
\label{sub:a_priori_estimates}

	With the choice of discrete function spaces and the formulation of the finite element problem \eqref{weakform_h} in Subsection~\ref{sub:discrete_spaces}, we continue to study the stability of the scheme. The analysis is similar to that presented in Subsection~\ref{sec:well_posedness} and we particularly emphasize the issues arising from the discretization in this separate presentation. 

	First, the incorporation of varying apertures requires some additional attention. For this purpose, we introduce the maximum value of $\epsilon$ on each element of the grid. More specifically, let us define $\epsilon_{e}$ as a piecewise constant function such that
	\begin{align}
		\epsilon_{e} &= \sup_{x \in e_\Omega} \epsilon(x) & \text{ on each } e_\Omega \in \mathcal{T}_{\Omega}. \label{def eps_Gamma}
	\end{align}
	By definition, this parameter equals one in $\Omega^n$.

	Secondly, for the purpose of the analysis, a specific discrete extension operator $\mathcal{R}_h : \mathit{\Lambda}_h \to \mathscr{V}_h$ is constructed similar to $\mathcal{R}$ from \eqref{Rtilde}. In particular, let the pair $(\mathcal{R}_h^{d-1} \lambda^{d-1}, p_\lambda^d) \in \bm{V}_h^d \times Q_h^d$ with $1 \le d \le n$ be the solution to the following problem:
		\begin{subequations} \label{Rhtilde}
			\begin{align}
				(K^{-1} \mathcal{R}_h^{d-1} \lambda^{d-1}, \bm{v}_0^d)_{\Omega^d} - 
				(p_\lambda^d, 
				\nabla \cdot \epsilon\bm{v}_0^d)_{\Omega^d} &= 0 &\forall \bm{v}_0^d &\in \bm{V}_{0,h}^d \label{Discrete Gradient} \\
				(\nabla \cdot \epsilon\mathcal{R}_h^{d-1} \lambda^{d-1}, q^d)_{\Omega^d} 
				+ (\epsilon_{e} p_{\lambda}^d, q^d)_{\Omega^d}&= 0 & \forall q^d &\in Q_h^d. \label{same as p_lambda}
			\end{align}
		\end{subequations}
	The corresponding boundary conditions are chosen to comply with the desired condition given in equation \eqref{eq: ExtensionR}, namely:
		\begin{subequations} \label{Rhtilde bcs}
			\begin{align}
				\mathcal{R}_h^{d-1} \lambda^{d-1} \cdot \bm{\nu} &= \hat{\Pi}_h^{d-1} \lambda^{d-1} &\text{on } &\Gamma^{d-1}, \\
				\mathcal{R}_h^{d-1} \lambda^{d-1} \cdot \bm{\nu} &= 0 &\text{on } &\partial \Omega^d \backslash \Gamma^{d-1}.
			\end{align}
		\end{subequations}

	The estimates on $\mathcal{R}_h^{d-1} \lambda^{d-1}$, analogous to Lemma~\ref{EllipticR} are given by the following lemma.
	\begin{lemma} \label{lem: Rhtilde}
	The solution $(\mathcal{R}_h^{d-1} \lambda^{d-1}, p_{\lambda}^d) \in \bm{V}_h^d \times Q_h^d$ to problem \eqref{Rhtilde} with boundary conditions given by \eqref{Rhtilde bcs} satisfies the following bounds:
		\begin{subequations} \label{Rhtilde bounds}
			\begin{align}
				\| K^{-\frac{1}{2}} \mathcal{R}_h^{d-1} \lambda^{d-1} \|_{L^2(\Omega^d)} + \| \epsilon_{e}^{\frac{1}{2}} p_{\lambda}^d \|_{L^2(\Omega^d)}
				&\lesssim \| \lambda^{d-1} \|_{L^2(\Gamma^{d-1})} \label{bound Rhlambda}\\
				\| \Pi_{Q_h}^d \nabla \cdot \epsilon \mathcal{R}_h^{d-1} \lambda^{d-1} \|_{L^2(\Omega^d)} &\lesssim \| \epsilon_{\max}^{\frac{1}{2}} \lambda^{d-1} \|_{L^2(\Gamma^{d-1})} \label{bound div Rhlambda}.
			\end{align}
		\end{subequations}
		with $\epsilon_{\max}|_{\Omega_i^d} = \| \epsilon \|_{L^{\infty}(\Omega_i^d)}$ and $\Pi_{Q_h}^d$ the $L^2$-projection onto $Q_h^d$.
	\end{lemma}
	\begin{proof}
		Let $\Pi_{V_h}^d$ be the Fortin interpolator related to $\bm{V}_h^d$ \cite{Boffi}. Moreover, let $\bm{v}_{\lambda, h}^d = \Pi_{V_h}^d \bm{v}_{\lambda}^d$ with $\bm{v}_{\lambda}^d \in \bm{V}^d$ such that
		\begin{align*}
			\bm{v}_{\lambda}^d \cdot \bm{\nu} &= \hat{\Pi}_h^{d-1} \lambda^{d-1} &\text{on } &\Gamma^{d-1}, \\
				\bm{v}_{\lambda}^d \cdot \bm{\nu} &= 0 &\text{on } &\partial \Omega^d \backslash \Gamma^{d-1},
		\end{align*}
		while also satisfying for some $s > 0$ (see \cite{quarteroni1999domain})
		\begin{align*}
			\| \bm{v}_{\lambda}^d \|_{H^s(\Omega^d)} + \| \nabla \cdot \bm{v}_{\lambda}^d \|_{L^2(\Omega^d)}
				&\lesssim \| \lambda^{d-1} \|_{L^2(\Gamma^{d-1})}.
		\end{align*}
		It follows that 
		\begin{align*}
			\bm{v}_{\lambda, h}^d \cdot \bm{\nu} 
			&= (\Pi_{V_h}^d \bm{v}_{\lambda}^d) \cdot \bm{\nu}
			= \hat{\Pi}_h^{d-1} (\bm{v}_{\lambda}^d \cdot \bm{\nu})
			= \hat{\Pi}_h^{d-1} \lambda^{d-1} &\text{on } &\Gamma^{d-1}.
		\end{align*}
		Hence, we may set the test function $\bm{v}_0^d = \mathcal{R}_h^{d-1} \lambda^{d-1} - \bm{v}_{\lambda, h}^d \in \bm{V}_{0, h}$. By continuity of the interpolator $\Pi_{V_h}^d$, see \cite{Arbogast}, 
		\begin{align} \label{bound a}
			\| \bm{v}_{\lambda, h}^d \|_{L^2(\Omega^d)} = 
			\| \Pi_{V_h}^d \bm{v}_{\lambda}^d \|_{L^2(\Omega^d)} \lesssim
			\| \bm{v}_{\lambda}^d \|_{H^s(\Omega^d)} + \| \nabla \cdot \bm{v}_{\lambda}^d \|_{L^2(\Omega^d)}
			&\lesssim \| \lambda^{d-1} \|_{L^2(\Gamma^{d-1})}
		\end{align}
		Furthermore, the interpolator has the property $\nabla \cdot \Pi_{V_h}^d \bm{v}_{\lambda}^d = \Pi_{Q_h}^d \nabla \cdot \bm{v}_{\lambda}^d$, with $\Pi_{Q_h}^d$ the $L^2$-projection onto $Q_h^d$. From this, we obtain
		\begin{align} \label{bound b}
			\| \nabla \cdot \bm{v}_{\lambda, h}^d \|_{L^2(\Omega^d)}
			= \| \nabla \cdot \Pi_{V_h}^d \bm{v}_{\lambda}^d \|_{L^2(\Omega^d)} 
			\le \| \nabla \cdot \bm{v}_{\lambda}^d \|_{L^2(\Omega^d)}
			&\lesssim \| \lambda^{d-1} \|_{L^2(\Gamma^{d-1})}
		\end{align}
		Now, let us set the test functions in \eqref{Rhtilde} as $\bm{v}_0^d = \mathcal{R}_h^{d-1} \lambda^{d-1} - \bm{v}_{\lambda, h}^d$ and $q^d = p_{\lambda}^d$. This gives us, as in equation \eqref{first bound cont}:
		\begin{align*}
			\| K^{-\frac{1}{2}} \mathcal{R}_h^{d-1} \lambda^{d-1} \|_{L^2(\Omega^d)}^2 +& \| \epsilon_{e}^{\frac{1}{2}} p_{\lambda}^d \|_{L^2(\Omega^d)}^2
			= (K^{-1} \mathcal{R}_h^{d-1} \lambda^{d-1}, \bm{v}_{\lambda, h}^d)_{\Omega^d} - 
				(p_\lambda^d, 
				\nabla \cdot \epsilon\bm{v}_{\lambda, h}^d)_{\Omega^d} \\
			&\lesssim 
			(\| K^{-\frac{1}{2}} \mathcal{R}_h^{d-1} \lambda^{d-1} \|_{L^2(\Omega^d)} + \| \epsilon_{e}^{\frac{1}{2}} p_{\lambda}^d \|_{L^2(\Omega^d)}) \
			\| \lambda^{d-1} \|_{L^2(\Gamma^{d-1})}.
		\end{align*}
		Here, we have used \eqref{bound a}, \eqref{bound b}, and the fact that $\epsilon(x) \le \epsilon_e(x)$ for all $x \in \Omega$. The first bound \eqref{bound Rhlambda} is now shown. Secondly, \eqref{bound div Rhlambda} follows by setting $q^d = \Pi_{Q_h}^d \nabla \cdot \epsilon \mathcal{R}_h^{d-1} \lambda^{d-1}$ and using \eqref{bound Rhlambda}:
		\begin{align*}
			\| \Pi_{Q_h}^d \nabla \cdot \epsilon \mathcal{R}_h^{d-1} \lambda^{d-1} \|_{L^2(\Omega^d)}
			\le \| \epsilon_e p_\lambda^d \|_{L^2(\Omega^d)}
			= \| \epsilon_e^{\frac{1}{2}} \epsilon_e^{\frac{1}{2}} p_\lambda^d \|_{L^2(\Omega^d)}
			&\le \| \epsilon_{\max}^{\frac{1}{2}} \lambda^{d-1} \|_{L^2(\Gamma^{d-1})}.
		\end{align*}
	\end{proof}

	We emphasize once more that this extension operator is only constructed for the sake of the analysis. Since we are continually interested in the combined flux $\bm{u}_0 + \mathcal{R}_h \lambda$ instead of the individual parts, it is generally more practical to choose $\mathcal{R}_h$ as any preferred extension operator which incorporates the essential boundary conditions.

	Let us continue by defining the norms in the discrete setting, which differ only slightly from the norms defined in \eqref{norms cont}.
	For $[\bm{v}_0, \mu] \in \mathscr{X}_h$, let us introduce the following norm:
		\begin{align}
			\| [\bm{v}_0, \mu ] \|_{\mathscr{X}_{\mathcal{R}, h}}^2 =& \ 
			\| K^{-\frac{1}{2}} (\bm{v}_0 + \mathcal{R}_h \mu) \|_{L^2(\Omega)}^2 
			+ \| \gamma^{\frac{1}{2}} K_{\nu}^{-\frac{1}{2}}\mu \|_{L^2(\Gamma)}^2 \nonumber\\
			&+ \| \Pi_{\mathscr{Q}_h} D \cdot [\epsilon (\bm{v}_0 + \mathcal{R}_h \mu), \hat{\epsilon} \mu ] \|_{L^2(\Omega)}^2. \label{VLnorms}
		\end{align}
	Here, $\Pi_{\mathscr{Q}_h}$ is the $L^2$-projection onto $\mathscr{Q}_h$. 
	The flexibility in the choice of $\mathcal{R}_h$ is apparent in this norm by noting that it depends on the combined flux, instead of its separate parts $\bm{u}_0$ and $\mathcal{R}_h \lambda$. The norm on the pressure $q \in \mathscr{Q}_h$ remains unchanged, and we recall it for convenience:
	\begin{align}
		\|q\|_{\mathscr{Q}_h} &= 
		\|q\|_{\mathscr{Q}} = \| \hat{\epsilon}_{\max} q \|_{L^2(\Omega)}. \label{Pnorm}
	\end{align}

	The discrete energy norm is formed as the combination of \eqref{VLnorms} and \eqref{Pnorm}:
	\begin{align}
		\enorm{(\bm{u}_0, \lambda, p)}_h^2 &= \| [\bm{u}_0, \lambda ] \|_{\mathscr{X}_{\mathcal{R}, h}}^2 + \| p \|_{\mathscr{Q}_h}^2. \label{eq: enormh}
	\end{align}

	Next, this energy norm is used to prove an inf-sup condition on $b_h$, as shown in the following Lemma.


	\begin{lemma}[Inf-Sup] \label{lem: inf-sup b}
		Let the bilinear form $b_h$ be defined by equation \eqref{operatorb_h} and let the function spaces $\mathscr{V}_{0,h}$, $\mathit{\Lambda}_h$, and $\mathscr{Q}_h$ comply with the restrictions from Subsection~\ref{sub:discrete_spaces}. Then there exists a constant $C_{b_h}>0$, independent of $\gamma$, $\epsilon$, and $h$ such that for any given function $p \in \mathscr{Q}_h$, 
		\begin{align}
			\sup_{(\bm{v}_0, \mu) \in \mathscr{X}_{h}} \frac{b_h(\bm{v}_0, \mu;\ p)}{\enorm{(\bm{v}_0, \mu, 0)}_h } &\geq C_{b_h} \enorm{(0,0, p)}_h.
		\end{align}
	\end{lemma}
	\begin{proof} 
		A similar strategy to that used in Lemma~\ref{lem: inf-sup b cont} is employed. First, the function $\mu^0 \in \Lambda^0_h$ is constructed. For each index $i$, recall that $j_{\max}$ denotes the index for which $\hat{\epsilon}_{i,j_{\max}} = \hat{\epsilon}_{\max}$. We then set $\mu_{ik}^0 = 0$ for $k \ne j_{\max}$ and choose
		\begin{align*}
			\mu_{i,j_{\max}}^0 = -\hat{\epsilon}_{\max} p^0.
		\end{align*}
		The following two properties then follow readily:
		\begin{align*}
			(\llbracket \hat{\epsilon} \mu^0 \rrbracket, p^0)_{\Omega^0}
			&= (-\hat{\epsilon}_{\max} \mu_{i,j_{\max}}^0, p^0)_{\Omega^0}
			= (-\mu_{i,j_{\max}}^0, \hat{\epsilon}_{\max} p^0)_{\Omega^0}
			= \| \hat{\epsilon}_{\max} p^0 \|^2
			, \nonumber \\
			\| \mu^0 \|_{L^2(\Gamma^0)}^2 &= \| \hat{\epsilon}_{\max} p^0 \|_{L^2(\Omega^0)}^2.
		\end{align*}
		Using a similar strategy, we construct $\mu^d$ with $1 \le d \le n-1$ such that $\mu_{ik}^d = 0$ for $k \ne j_{\max}$. The remaining function $\mu_{i,j_{\max}}^d$ is defined such that
		\begin{align}
			(\hat{\epsilon}_{\max} \mu_{i,j_{\max}}^d, \varphi_k)_{\Omega_i^d} 
			&= (- \hat{\epsilon}_{\max}^2 p^d + \nabla \cdot \epsilon \mathcal{R}_h^{d-1} \mu^{d-1}, \varphi_k)_{\Omega_i^d}. \label{constr mumax}
		\end{align}
		for all basis functions $\varphi_k \in \Lambda_{ij, h}^d$. We show that $\mu_{i,j_{\max}}^d$ is well-defined by rewriting it as the linear combination $\mu_{i,j_{\max}}^d = \sum_k \alpha_k \varphi_k$. The matrix for solving $\alpha_k$ is then given by $A_{kl} = (\hat{\epsilon}_{\max} \varphi_l, \varphi_k)_{\Omega_i^d}$ which is symmetric and positive definite given $\hat{\epsilon}_{\max} > 0$ by \eqref{eps_max bound}.

		Moreover, the chosen $\mu^d$ has the following properties where we use \eqref{Q_h equals jump Lambda_h} and the bounds \eqref{eps upperbounded by eps_max} and \eqref{bound div Rhlambda}.
		\begin{subequations}
			\begin{align}
				(\llbracket \hat{\epsilon} \mu^d \rrbracket, p^d)_{\Omega^d} 
				&= \| \hat{\epsilon}_{\max} p^d \|_{L^2(\Omega^d)}^2 - (\nabla \cdot \epsilon \mathcal{R}_h^{d-1} \mu^{d-1}, p^d)_{\Omega^d}, \\
				\| \mu^d \|_{L^2(\Gamma^d)} &\lesssim \| \hat{\epsilon}_{\max} p^d \|_{L^2(\Omega^d)} 
				+ \| \mu^{d-1} \|_{L^2(\Gamma^{d-1})}. \label{mu_h bound}
			\end{align}
		\end{subequations}

		The functions $\bm{v}_0^d$ with $1 \le d \le n$ now remain to be constructed in order to obtain additional control on the pressure. As in Lemma~\ref{lem: inf-sup b cont}, we set 
		\begin{align}
			\bm{v}_0^d &= 0 & \text{for } 1 &\le d \le n - 1. 
		\end{align}
		For the final case $d = n$, we recall that $Q_h^n \times V_{0,h}^n$ is a stable mixed finite element pair as given by \eqref{Fortin}. Keeping this in mind, $\bm{v}_0^n$ is constructed such that it forms the following solution together with $p_v^n \in Q_h^n$
		\begin{align*}
			(K^{-1} \bm{v}_0^n, \bm{w}_0^n)_{\Omega^n} -
			(p_v^n, 
			\nabla \cdot \bm{w}_0^n)_{\Omega^n} &= 0 &\bm{w}_0^n &\in \bm{V}_{0,h}^n \\
			(\nabla \cdot \bm{v}_0^n, q^n)_{\Omega^n} &= (p^n - \nabla \cdot \mathcal{R}_h^{n-1} \mu^{n - 1}, q^n)_{\Omega^n}, & q^n &\in Q_h^n\\
			\bm{v}_0^n \cdot \bf{\nu} &= 0, & \text{on } &\partial \Omega^n \backslash \partial \Omega_D.
		\end{align*}

		We note that $\epsilon = \hat{\epsilon}_{\max} = 1$ in $\Omega^n$ and it follows by construction that
		\begin{align}
			-b_h(\bm{v}_0, \mu;\ p) &= ( p, \nabla \cdot \epsilon(\bm{v}_0 + \mathcal{R}_h \mu))_{\Omega}
			+ ( p, \llbracket \hat{\epsilon} \mu \rrbracket )_{\Omega} \nonumber \\
			&= \| \hat{\epsilon}_{\max} p \|_{L^2(\Omega)}^2 
			= \| p \|_{\mathscr{Q}_h}^2. \label{result1disc}
		\end{align}

		The corresponding bounds on $\bm{v}_0^n$ are derived using standard mixed finite element arguments and \eqref{bound div Rhlambda}:
		\begin{align}
			\| K^{-\frac{1}{2}} \bm{v}_0^n \|_{L^2(\Omega^n)} 
			+ \| \Pi_{Q_h}^n \nabla \cdot \epsilon \bm{v}_0^n \|_{L^2(\Omega^n)} &= 
			\| K^{-\frac{1}{2}} \bm{v}_0^n \|_{L^2(\Omega^n)} 
			+ \| \nabla \cdot \bm{v}_0^n \|_{L^2(\Omega^n)} \nonumber\\
			&\lesssim
			\| p^n \|_{L^2(\Omega^n)} 
			+ \| \nabla \cdot \mathcal{R}_h^{n-1} \mu^{n - 1} \|_{L^2(\Gamma^{n-1})} \nonumber\\
			&\lesssim
			\| p^n \|_{L^2(\Omega^n)} 
			+ \| \mu^{n - 1} \|_{L^2(\Gamma^{n-1})} \label{v_h bound}
		\end{align}

		The construction of $(\bm{v}_0,\mu)$ is now complete and the bounds \eqref{mu_h bound} and \eqref{v_h bound} in combination with \eqref{Rhtilde bounds} give us
			\begin{align}
				\enorm{(\bm{v}_0, \mu, 0)}_h^2 
				\lesssim & \ 
				\| K^{-\frac{1}{2}} \bm{v}_0 \|_{L^2(\Omega)}^2
				+ \| K^{-\frac{1}{2}} \mathcal{R}_h \mu \|_{L^2(\Omega)}^2 
				+ \| \gamma^{\frac{1}{2}} K_{\nu}^{-\frac{1}{2}}\mu \|_{L^2(\Gamma)}^2 \nonumber\\
				&+ \| \Pi_{\mathscr{Q}_h} (\nabla \cdot \epsilon (\bm{v}_0 + \mathcal{R}_h \mu)
				+ \llbracket \hat{\epsilon} \mu \rrbracket) \|_{L^2(\Omega)}^2 \nonumber\\
				\lesssim& \ 
				\| \hat{\epsilon}_{\max} p \|_{L^2(\Omega)}^2 
				= \| p \|_{\mathscr{Q}_h}^2. \label{result2disc}
			\end{align}

		The proof is concluded by combining \eqref{result1disc} and \eqref{result2disc}.
	\end{proof}

	With the previous lemma, we are ready to present the stability result, given by the following theorem.
	\begin{theorem}[Stability] \label{Well-posedness discrete}
		Let the mesh and function spaces $\mathscr{V}_{0,h}$, $\mathit{\Lambda}_h$, and $\mathscr{Q}_h$ be chosen such that they comply with the restrictions from Subsection~\ref{sub:discrete_spaces}. Then the discrete problem \eqref{weakform_h} has a unique solution satisfying the stability estimate
		\begin{align}
			\enorm{(\bm{u}_0, \lambda, p)}_h \lesssim \| \epsilon^{\frac{3}{2}} f \|_{L^2(\Omega)} + \| g \|_{H^{\frac{1}{2}}(\partial \Omega_D)}. \label{estimate disc}
		\end{align}
	\end{theorem}
	\begin{proof}
		Starting with Lemma~\ref{lem: inf-sup b}, let $(\bm{u}_{0, p}, \lambda_p)$ be the constructed pair based on the pressure distribution $p$ with the following two properties
		\begin{subequations}\label{eq: u_p}
			\begin{align}
				- ( p, D \cdot [\epsilon (\bm{u}_{0, p} + \mathcal{R}_h \lambda_p), \hat{\epsilon} \lambda_p])_{\Omega} &= \| p \|_{\mathscr{Q}_h}^2, \\
				\| [\bm{u}_{0, p}, \lambda_p] \|_{\mathscr{X}_{\mathcal{R}, h}} &\lesssim \| p \|_{\mathscr{Q}_h}. 
			\end{align}
		\end{subequations}
		We then introduce the following test functions with $\delta_1 > 0$ a constant to be determined later:
		\begin{align*}
			\bm{v}_0 &= \bm{u}_0 + \delta_1 \bm{u}_{0, p}, &
			\mu &= \lambda + \delta_1 \lambda_p, &
			q &= p + \Pi_{\mathscr{Q}_h} D \cdot [\epsilon (\bm{u}_0 + \mathcal{R}_h \lambda), \hat{\epsilon} \lambda ].
		\end{align*}

		Substitution of these test functions in \eqref{weakform_h} gives us
		\begin{align*}
			\| K^{-\frac{1}{2}} &(\bm{u}_0 + \mathcal{R}_h \lambda) \|_{L^2(\Omega)}^2 
			+ \| \gamma^{\frac{1}{2}} K_{\nu}^{-\frac{1}{2}}\lambda \|_{L^2(\Gamma)}^2
			+ \| \Pi_{\mathscr{Q}_h} D \cdot [\epsilon (\bm{u}_0 + \mathcal{R}_h \lambda), \hat{\epsilon} \lambda ] \|_{L^2(\Omega)}^2 \nonumber\\
			&+ \delta_1 \| p \|_{\mathscr{Q}_h}^2 \nonumber\\
			=&\ 
			-\langle g, \epsilon (\bm{u}_0 + \delta_1 \bm{u}_{0, p}) \cdot \bm{\nu} \rangle_{\partial \Omega_D} 
			+ (\epsilon^2 f, p + \Pi_{\mathscr{Q}_h} D \cdot [\epsilon (\bm{u}_0 + \mathcal{R}_h \lambda), \hat{\epsilon} \lambda ])_{\Omega} \nonumber \\
			&- (K^{-1} (\bm{u}_0 + \mathcal{R}_h \lambda), \delta_1 (\bm{u}_{0, p} + \mathcal{R}_h \lambda_p))_{\Omega}
			- \langle \frac{\gamma}{K_{\nu}} \lambda, \delta_1 \lambda_p \rangle_{\Gamma} \nonumber\\
			\le&\ 
			\frac{1}{2 \delta_2} \| g \|_{H^{\frac{1}{2}}(\partial \Omega_D)}^2 
			+ \frac{\delta_2}{2} \| \epsilon (\bm{u}_0 + \delta_1 \bm{u}_{0, p}) \cdot \bm{\nu} \|_{H^{-\frac{1}{2}}(\partial \Omega_D)}^2 
			+ \big(\frac{1}{2 (\delta_1 \hat{\epsilon}_{\max})^2} + \frac{1}{2}\big)
			\|\epsilon^2 f \|_{L^2(\Omega)}^2 \nonumber\\
			&+ \frac{\delta_1^2}{2} \| \hat{\epsilon}_{\max} p \|_{L^2(\Omega)}^2 
			+ \frac{1}{2} \| \Pi_{\mathscr{Q}_h} D \cdot [\epsilon (\bm{u}_0 + \mathcal{R}_h \lambda), \hat{\epsilon} \lambda ] \|_{L^2(\Omega)}^2 \nonumber \\
			&+ \frac{1}{2} \| K^{-\frac{1}{2}} (\bm{u}_0 + \mathcal{R}_h \lambda) \|_{L^2(\Omega)}^2 
			+ \frac{\delta_1^2}{2} \| K^{-\frac{1}{2}} (\bm{u}_{0, p} + \mathcal{R}_h \lambda_p) \|_{L^2(\Omega)}^2) \nonumber\\
			&+ \frac{1}{2} \| \gamma^{\frac{1}{2}} K_{\nu}^{-\frac{1}{2}}\lambda \|_{L^2(\Gamma)}^2
			+ \frac{\delta_1^2}{2} \| \gamma^{\frac{1}{2}} K_{\nu}^{-\frac{1}{2}}\lambda_p \|_{L^2(\Gamma)}^2,
		\end{align*}
		with $\delta_2 > 0$ a constant. 

		Let us consider the second term after the inequality. The fact that the extension operator $\mathcal{R}_h$ has zero normal trace on $\partial \Omega^d$, the positive definiteness of $K$, and the trace theorem give us
		\begin{align}
			\| \epsilon (\bm{u}_0 + \delta_1 \bm{u}_{0, p}) \cdot \bm{\nu} &\|_{H^{-\frac{1}{2}}(\partial \Omega_D)}
			\le
			\| \epsilon (\bm{u}_0 + \mathcal{R}_h \lambda) \cdot \bm{\nu} \|_{H^{-\frac{1}{2}}(\partial \Omega_D)}
			+ \delta_1 \| \epsilon (\bm{u}_{0, p} + \mathcal{R}_h \lambda_p) \cdot \bm{\nu} \|_{H^{-\frac{1}{2}}(\partial \Omega_D)} \nonumber\\
			\lesssim&\ 
			\| K^{-\frac{1}{2}} (\bm{u}_0 + \mathcal{R}_h \lambda) \|_{L^2(\Omega)}
			+\| \nabla \cdot \epsilon (\bm{u}_0 + \mathcal{R}_h \lambda) \|_{L^2(\Omega)}  \nonumber\\
			&+  \delta_1 \| K^{-\frac{1}{2}} (\bm{u}_{0, p} + \mathcal{R}_h \lambda_p) \|_{L^2(\Omega)}
			+ \delta_1 \| \nabla \cdot \epsilon (\bm{u}_{0, p} + \mathcal{R}_h \lambda_p) \|_{L^2(\Omega)}.
		\end{align}
		Considering the second term, let $\epsilon_h$ be the piecewise constant approximation of $\epsilon$. We then use $\nabla \cdot \mathscr{V}_h \subseteq \mathscr{Q}_h$ from \eqref{Fortin} and the $L^{\infty}$ approximation property of $\epsilon_h$ from \cite{Douglas1975} to obtain
		\begin{align}
			\| \nabla \cdot \epsilon (\bm{u}_0 + \mathcal{R}_h &\lambda) \|_{L^2(\Omega)}
			\le\ \| (\nabla \epsilon) \cdot (\bm{u}_0 + \mathcal{R}_h \lambda) \|_{L^2(\Omega)}
			+ \| \epsilon \nabla \cdot (\bm{u}_0 + \mathcal{R}_h \lambda) \|_{L^2(\Omega)} \nonumber\\
			\lesssim&\ \| \epsilon^{\frac{1}{2}} \|_{L^{\infty}(\Omega)} \| \bm{u}_0 + \mathcal{R}_h \lambda \|_{L^2(\Omega)}
			+ \| (\epsilon - \epsilon_h) \nabla \cdot (\bm{u}_0 + \mathcal{R}_h \lambda) \|_{L^2(\Omega)} \nonumber\\
			&+ \| \Pi_{\mathscr{Q}_h} \epsilon_h \nabla \cdot (\bm{u}_0 + \mathcal{R}_h \lambda) \|_{L^2(\Omega)} \nonumber\\
			\lesssim&\ \| \bm{u}_0 + \mathcal{R}_h \lambda \|_{L^2(\Omega)}
			+ \| (\epsilon - \epsilon_h) \nabla \cdot (\bm{u}_0 + \mathcal{R}_h \lambda) \|_{L^2(\Omega)} \nonumber\\
			&+ \| \Pi_{\mathscr{Q}_h} (\epsilon_h - \epsilon) \nabla \cdot (\bm{u}_0 + \mathcal{R}_h \lambda) \|_{L^2(\Omega)}
			+ \| \Pi_{\mathscr{Q}_h} \epsilon \nabla \cdot (\bm{u}_0 + \mathcal{R}_h \lambda) \|_{L^2(\Omega)} \nonumber\\
			\lesssim&\ \| \bm{u}_0 + \mathcal{R}_h \lambda \|_{L^2(\Omega)}
			+ h\| \nabla \epsilon \|_{L^{\infty}(\Omega)} \| \nabla \cdot (\bm{u}_0 + \mathcal{R}_h \lambda) \|_{L^2(\Omega)} \nonumber\\
			&+ \| \Pi_{\mathscr{Q}_h} \epsilon \nabla \cdot (\bm{u}_0 + \mathcal{R}_h \lambda) \|_{L^2(\Omega)} \nonumber\\
			\lesssim&\ \| \bm{u}_0 + \mathcal{R}_h \lambda \|_{L^2(\Omega)}
			+ \| \Pi_{\mathscr{Q}_h} \epsilon \nabla \cdot (\bm{u}_0 + \mathcal{R}_h \lambda) \|_{L^2(\Omega)}, \label{eq: step 1}
		\end{align}
		using an inverse inequality.
		Finally, we use assumption \eqref{gamma_0} and the positive definiteness of $K$ to derive
		\begin{align}
			\| \Pi_{\mathscr{Q}_h} \epsilon \nabla \cdot (\bm{u}_0 + \mathcal{R}_h \lambda) \|_{L^2(\Omega)}
			\lesssim&\
			\| \Pi_{\mathscr{Q}_h} (\nabla \epsilon) \cdot (\bm{u}_0 + \mathcal{R}_h \lambda) \|_{L^2(\Omega)}
			+ \| \Pi_{\mathscr{Q}_h} \nabla \cdot \epsilon (\bm{u}_0 + \mathcal{R}_h \lambda) \|_{L^2(\Omega)} \nonumber\\
			\lesssim&\ \| \bm{u}_0 + \mathcal{R}_h \lambda \|_{L^2(\Omega)}
			+ \| \Pi_{\mathscr{Q}_h} \nabla \cdot \epsilon (\bm{u}_0 + \mathcal{R}_h \lambda) \|_{L^2(\Omega)} \nonumber\\
			\lesssim&\ \| \bm{u}_0 + \mathcal{R}_h \lambda \|_{L^2(\Omega)}
			+ \| \Pi_{\mathscr{Q}_h} D \cdot [\epsilon (\bm{v}_0 + \mathcal{R}_h \mu), \hat{\epsilon} \mu ] \|_{L^2(\Omega)} \nonumber\\
			&\ + \| \Pi_{\mathscr{Q}_h} \llbracket \hat{\epsilon} \mu \rrbracket \|_{L^2(\Omega)} \nonumber\\
			\lesssim&\ \| [\bm{u}_0, \lambda ] \|_{\mathscr{X}_{\mathcal{R}, h}}. \label{eq: step 2}
		\end{align}

		The steps from \eqref{eq: step 1} and \eqref{eq: step 2} are then repeated for $\bm{u}_{0, p} + \mathcal{R}_h \lambda_p$ and we conclude
		\begin{align}
			\| \epsilon (\bm{u}_0 + \delta_1 \bm{u}_{0, p}) \cdot \bm{\nu} &\|_{H^{-\frac{1}{2}}(\partial \Omega_D)}^2
			\lesssim \| [\bm{u}_0, \lambda ] \|_{\mathscr{X}_{\mathcal{R}, h}}^2 + \delta_1^2 \| [\bm{u}_{0, p}, \lambda_p ] \|_{\mathscr{X}_{\mathcal{R}, h}}^2.
		\end{align}
		By setting $\delta_2$ sufficiently small and using the properties of $(\bm{u}_{0, p}, \lambda_p)$ from \eqref{eq: u_p}, we obtain
		\begin{align}
			\| [\bm{u}_0, \lambda ] \|_{\mathscr{X}_{\mathcal{R}, h}}^2 
			+ \delta_1 \| p \|_{\mathscr{Q}_h}^2 
			&\lesssim 
			\| \epsilon^{\frac{3}{2}} f \|_{L^2(\Omega)}^2
			+ \| g \|_{H^{\frac{1}{2}}(\partial \Omega_D)}^2
			+ \delta_1^2 \| p \|_{\mathscr{Q}_h}^2 
		\end{align}
		Choosing a sufficiently small value for $\delta_1$ then concludes the stability estimate. Since we are considering a square linear system, this estimate implies existence and uniqueness of the solution.
	\end{proof}

	With the stability result from Theorem~\ref{Well-posedness discrete}, we continue with the basic error estimates. The true solution, i.e. the unique solution to \eqref{eq: weakform}, will be denoted by $(\bm{u}_0, \lambda, p)$
	and the finite element solution will be called $(\bm{u}_{0,h}, \lambda_h, p_h)$. Since we are interested in the combined fluxes, we re-introduce
	\begin{align*}
		\bm{u} &= \bm{u}_{0} + \mathcal{R} \lambda, & 
		\bm{u}_h &= \bm{u}_{0, h} + \mathcal{R}_h \lambda_h.
	\end{align*}

	These definitions show the flexibility in the choice of extension operator. In fact, for a given $\bm{u}$ with normal trace $\lambda$ on $\Gamma$, the above decomposition is possible for an arbitrary choice of $\mathcal{R}$ satisfying \eqref{eq: ExtensionR}. In turn, the goal of the analysis which follows is not to prove that $\bm{u}_{0,h}$ converges to $\bm{u}_0$ since this depends completely on the choice of extension operator. Rather, we aim to show that the combined flux $\bm{u}_h$ converges to $\bm{u}$. To emphasize this nuance, we introduce the norm:
	\begin{align}
		\| [\bm{v}, \mu] \|_{\mathscr{X}_h}^2 =& \ 
			\| K^{-\frac{1}{2}} \bm{v} \|_{L^2(\Omega)}^2 
			+ \| \gamma^{\frac{1}{2}} K_{\nu}^{-\frac{1}{2}} \mu \|_{L^2(\Gamma)}^2
			+ \| \Pi_{\mathscr{Q}_h} D \cdot [\epsilon \bm{v}, \hat{\epsilon} \mu ] \|_{L^2(\Omega)}^2. \label{norm on u}
	\end{align}

	Let $\Pi_{\mathit{\Lambda}_h}: \mathit{\Lambda} \to \mathit{\Lambda}_h$ and $\Pi_{\mathscr{Q}_h}: \mathscr{Q} \to \mathscr{Q}_h$ be $L^2$-projection operators to the corresponding discrete spaces. Additionally, let $\Pi_{V_h}^d: V^d \cap (L^{2 + s})^d \to V_h^d$ for $1 \le d \le n$ and $s>0$ denote the standard Fortin interpolator associated with the chosen flux space $V^d$. The direct sum of $\Pi_{V_h}^d$ over all dimensions $1 \le d \le n$ gives us $\Pi_{\mathscr{V}_h}$.

	Let $k$ represent the order of the polynomials in the pressure space. The following interpolation estimates hold for the operators $\Pi_{\mathscr{V}_h}$, $\Pi_{\mathit{\Lambda}_h}$, and $\Pi_{\mathscr{Q}_h}$ and a chosen value of $k$ (see e.g. \cite{Arbogast,Boffi}):
	\begin{subequations} \label{eqs proj est}
		\begin{align}
			\| \bm{u} - \Pi_{\mathscr{V}_h} \bm{u} \|_{0,\Omega} &\lesssim \| \bm{u} \|_{r,\Omega} ~ h^r, 
			& 1 &\leq r \leq k+1, \label{proj est u}\\
			\| \nabla \cdot (\bm{u} - \Pi_{\mathscr{V}_h} \bm{u}) \|_{0,\Omega} &\lesssim \| \nabla \cdot \bm{u} \|_{r,\Omega} ~ h^r,
			& 1 &\leq r \leq k+1, \label{proj est divu}\\
			\| \lambda - \Pi_{\mathit{\Lambda}_h} \lambda \|_{0, \Gamma} &\lesssim \| \lambda \|_{r, \Gamma} ~ h^r,
			& 1 &\leq r \leq k+1, \label{proj est lambda}\\
			\| p - \Pi_{\mathscr{Q}_h} p \|_{0,\Omega} &\lesssim\| p \|_{r,\Omega} ~ h^r,
			& 1 &\leq r \leq k+1. \label{proj est p}
		\end{align}
	\end{subequations}
	Here, $\|\cdot\|_{r, \Sigma}$ is short-hand for the $H^r(\Sigma)$-norm. 

	We are now ready to continue with the error estimates. For this, we employ the same strategy as in \cite{arbogast2016linear}.
	First, the test functions are chosen from the discrete function spaces and we subtract the systems \eqref{eq: weakform} and \eqref{weakform_h} to obtain
	\begin{align}
			(K^{-1} (\bm{u} - \bm{u}_h), \bm{v}_h)_{\Omega}
			&+ \langle \frac{\gamma}{K_{\nu}} (\lambda - \lambda_h), \mu_h \rangle_{\Gamma}
			- ( p - p_h, D \cdot [\epsilon \bm{v}_h, \hat{\epsilon} \mu_h])_{\Omega} \nonumber \\
			&+ ( q_h, D \cdot [\epsilon (\bm{u} - \bm{u}_h), \hat{\epsilon} (\lambda - \lambda_h)])_{\Omega}
			= 0. \label{sys subtracted}
	\end{align}

	An immediate consequence of choosing $q_h = \Pi_{\mathscr{Q}_h} D \cdot [\epsilon (\bm{u} - \bm{u}_h), \hat{\epsilon} (\lambda - \lambda_h)]$ is that
	\begin{align}
		\Pi_{\mathscr{Q}_h} D \cdot [\epsilon (\bm{u} - \bm{u}_h), \hat{\epsilon} (\lambda - \lambda_h)] = 0. \label{eq: perfect div}
	\end{align}
	
	Turning back to \eqref{sys subtracted}, we introduce the projections of the true solution onto the corresponding spaces and manipulate the equation to
	\begin{align*}
		(K^{-1} (\Pi_{\mathscr{V}_h} \bm{u} &- \bm{u}_h), \bm{v}_h)_{\Omega}
		+ \langle \frac{\gamma}{K_{\nu}} (\Pi_{\mathit{\Lambda}_h} \lambda - \lambda_h), \mu_h \rangle_{\Gamma}
		- ( \Pi_{\mathscr{Q}_h} p - p_h, D \cdot [\epsilon \bm{v}_h, \hat{\epsilon} \mu_h])_{\Omega} \nonumber \\
		&+ ( q_h, \Pi_{\mathscr{Q}_h} D \cdot [\epsilon (\Pi_{\mathscr{V}_h} \bm{u} - \bm{u}_h), \hat{\epsilon} (\Pi_{\mathit{\Lambda}_h} \lambda - \lambda_h)])_{\Omega} \nonumber \\
		=&\ (K^{-1} (\Pi_{\mathscr{V}_h} \bm{u} - \bm{u}), \bm{v}_h)_{\Omega}
		+ \langle \frac{\gamma}{K_{\nu}} (\Pi_{\mathit{\Lambda}_h} \lambda - \lambda), \mu_h \rangle_{\Gamma} 
		- ( \Pi_{\mathscr{Q}_h} p - p, D \cdot [\epsilon \bm{v}_h, \hat{\epsilon} \mu_h])_{\Omega} \nonumber \\
		&+ ( q_h, \Pi_{\mathscr{Q}_h} D \cdot [\epsilon (\Pi_{\mathscr{V}_h} \bm{u} - \bm{u}), \hat{\epsilon} (\Pi_{\mathit{\Lambda}_h} \lambda - \lambda)])_{\Omega}.
	\end{align*}
	
	We continue by making the following explicit choice of test functions. For that, we first introduce the pair $(\bm{u}_{p, h}, \lambda_{p, h})$ from the inf-sup condition in Lemma~\ref{lem: inf-sup b} based on the pressure distribution $\Pi_{\mathscr{Q}_h} p - p_h$. Let us recall the following two properties
	\begin{subequations} 
		\begin{align}
			- ( \Pi_{\mathscr{Q}_h} p - p_h, D \cdot [\epsilon \bm{u}_{p, h}, \hat{\epsilon} \lambda_{p, h}])_{\Omega} &= \| \Pi_{\mathscr{Q}_h} p - p_h \|_{\mathscr{Q}_h}^2, \label{eq: u_p,h eq}\\
			\| [\bm{u}_{p, h}, \lambda_{p, h}] \|_{\mathscr{X}_h} &\lesssim \| \Pi_{\mathscr{Q}_h} p - p_h \|_{\mathscr{Q}_h}.\label{eq: u_p,h ineq}
		\end{align}
	\end{subequations}

	Under the assumption that the solution has sufficient regularity, we are ready to set the test functions as
	\begin{align*}
		\bm{v}_h &= \Pi_{\mathscr{V}_h} \bm{u} - \bm{u}_h + \delta_1 \bm{u}_{p, h}, \\
		\mu_h &= \Pi_{\mathit{\Lambda}_h} \lambda - \lambda_h + \delta_1 \lambda_{p, h}, \\
		q_h &= \Pi_{\mathscr{Q}_h} p - p_h,
	\end{align*}
	with $\delta_1 > 0$ to be determined later.	
	Substitution in the above system and applying Cauchy-Schwarz and Young inequalities multiple times (with parameters ${\delta_1, \delta_2, \delta_3 > 0}$) and \eqref{eq: u_p,h eq} then gives us
	\begin{align}
		\left(1 - \frac{\delta_2}{2} - \frac{\delta_3}{2} \right) & \left( \| K^{-\frac{1}{2}} (\Pi_{\mathscr{V}_h} \bm{u} - \bm{u}_h) \|_{L^2(\Omega)}^2
		+ \| \gamma^{\frac{1}{2}} K_{\nu}^{-\frac{1}{2}} (\Pi_{\mathit{\Lambda}_h} \lambda - \lambda_h) \|_{L^2(\Gamma)}^2 \right) \nonumber\\
		&+ \delta_1 \| \Pi_{\mathscr{Q}_h} p - p_h \|_{\mathscr{Q}_h}^2 \nonumber\\
		\le& \ \left(\frac{1}{2\delta_3} + \frac{1}{2} \right) \Big(\| K^{-\frac{1}{2}} (\Pi_{\mathscr{V}_h} \bm{u} - \bm{u}) \|_{L^2(\Omega)}^2
		+ \| \gamma^{\frac{1}{2}} K_{\nu}^{-\frac{1}{2}} (\Pi_{\mathit{\Lambda}_h} \lambda - \lambda) \|_{L^2(\Gamma)}^2 \Big) \nonumber\\
		&+ \frac{1}{2\delta_1} \| \hat{\epsilon}_{\max}^{-1} \Pi_{\mathscr{Q}_h} D \cdot [\epsilon (\Pi_{\mathscr{V}_h} \bm{u} - \bm{u}), \hat{\epsilon} (\Pi_{\mathit{\Lambda}_h} \lambda - \lambda)] \|_{L^2(\Omega)}^2\nonumber\\
		&+ \frac{\delta_1}{2}\| \hat{\epsilon}_{\max} (\Pi_{\mathscr{Q}_h} p - p_h ) \|_{L^2(\Omega)}^2 \nonumber\\
		&+ \left(\frac{1}{2\delta_2} + \frac{1}{2} \right) \delta_1^2 \left(\| K^{-\frac{1}{2}} \bm{u}_{p, h} \|_{L^2(\Omega)}^2
		 + \| \gamma^{\frac{1}{2}} K_{\nu}^{-\frac{1}{2}} \lambda_{p, h} \|_{L^2(\Gamma)}^2  \right)\nonumber\\
		&+ ( p - \Pi_{\mathscr{Q}_h} p, D \cdot [\epsilon (\Pi_{\mathscr{V}_h} \bm{u} - \bm{u}_h + \delta_1 \bm{u}_{p, h}), \hat{\epsilon} (\Pi_{\mathit{\Lambda}_h} \lambda - \lambda_h + \delta_1 \lambda_{p, h})])_{\Omega}. \label{firstestimate}
	\end{align}

	We continue to form a bound on the last term in \eqref{firstestimate}. For brevity, we briefly revert to the notation of $\bm{v}_h$ and $\mu_h$. The definition of the operator $D \cdot$ and the product rule give us
	\begin{align}
		( p - \Pi_{\mathscr{Q}_h} p &, D \cdot [\epsilon \bm{v}_h, \hat{\epsilon} \mu_h])_{\Omega}
		= (p - \Pi_{\mathscr{Q}_h} p, \nabla \cdot \epsilon \bm{v}_h + \llbracket \hat{\epsilon} \mu_h \rrbracket)_{\Omega} \nonumber\\
		&= (p - \Pi_{\mathscr{Q}_h} p, \nabla \epsilon \cdot \bm{v}_h)_{\Omega}
		+ (p - \Pi_{\mathscr{Q}_h} p, \epsilon \nabla \cdot \bm{v}_h)_{\Omega}
		+ (p - \Pi_{\mathscr{Q}_h} p, \llbracket \hat{\epsilon} \mu_h \rrbracket)_{\Omega}. \label{eq: er est 1}
	\end{align}

	Let us consider the three terms on the right-hand side one at a time. For the first term, we use Cauchy-Schwarz, \eqref{grad eps bound}, and the positive-definiteness of $K$ to derive
	\begin{align}
		(p - \Pi_{\mathscr{Q}_h} p, \nabla \epsilon \cdot \bm{v}_h)_{\Omega} 
		&\le \| p - \Pi_{\mathscr{Q}_h} p \|_{L^2(\Omega)} \| \nabla \epsilon \|_{L^\infty(\Omega)} \| \bm{v}_h \|_{L^2(\Omega)} \nonumber\\
		&\lesssim \| p - \Pi_{\mathscr{Q}_h} p \|_{L^2(\Omega)} \| \bm{v}_h \|_{L^2(\Omega)} \nonumber\\
		&\lesssim \| p - \Pi_{\mathscr{Q}_h} p \|_{L^2(\Omega)} \| K^{-\frac{1}{2}} \bm{v}_h \|_{L^2(\Omega)}.
	\end{align}

	Let us continue with the second term. Let $\epsilon_h$ be the piecewise constant approximation of $\epsilon$. Since $\nabla \cdot \bm{v}_h \in \mathscr{Q}_h$, we have $\epsilon_h \nabla \cdot \bm{v}_h \in \mathscr{Q}_h$. 
	We use this in combination with the $L^{\infty}$ approximation property of $\epsilon_h$ from \cite{Douglas1975} and an inverse inequality to derive
	\begin{align}
		(p - \Pi_{\mathscr{Q}_h} p, \epsilon \nabla \cdot \bm{v}_h)_{\Omega}
		&= ((I - \Pi_{\mathscr{Q}_h}) p, (\epsilon - \epsilon_h) \nabla \cdot \bm{v}_h)_{\Omega} \nonumber \\
		&\lesssim \| p - \Pi_{\mathscr{Q}_h} p \|_{L^2(\Omega)} h \| \nabla \epsilon \|_{L^{\infty}(\Omega)} \|\nabla \cdot \bm{v}_h \|_{L^2(\Omega)} \nonumber \\
		&\lesssim \| p - \Pi_{\mathscr{Q}_h} p \|_{L^2(\Omega)} \| \bm{v}_h \|_{L^2(\Omega)} \nonumber \\
		&\lesssim \| p - \Pi_{\mathscr{Q}_h} p \|_{L^2(\Omega)} \| K^{-\frac{1}{2}} \bm{v}_h \|_{L^2(\Omega)}.   \label{eq: er est 2}
	\end{align}

	Next, we consider the final term in \eqref{eq: er est 1}. With the exception of $d=1$ and $n=3$, this term is zero since $\llbracket \mathit{\Lambda}_h \rrbracket = \mathscr{Q}_h$ by \eqref{Q_h equals jump Lambda_h} and $\hat{\epsilon}$ is constant. 
	Thus, let us consider a $\Omega^d$ with $d = 1$ and $n=3$. In this case, we derive
	\begin{align}
		(p - \Pi_{\mathscr{Q}_h} p, \llbracket \hat{\epsilon} \mu_h \rrbracket)_{\Omega^d}
		&\lesssim  \| p - \Pi_{\mathscr{Q}_h} p \|_{L^2(\Omega^d)} \| \hat{\epsilon} \mu_h \|_{L^2(\Omega^d)} \nonumber \\
		&\lesssim  \| p - \Pi_{\mathscr{Q}_h} p \|_{L^2(\Omega^d)} \| \gamma^{\frac{1}{2}} K_{\nu}^{-\frac{1}{2}} \mu_h \|_{L^2(\Omega^d)}. \label{eq: er est 3}
	\end{align}	
	The final equality follows by noting that $\hat{\epsilon}^1 \eqsim \epsilon^2 \eqsim \gamma^{\frac{3 - 2}{2}}$ and 
	that there is no extension of the one-dimensional domain beyond the point where $\gamma = 0$. 

	For the final term in \eqref{firstestimate}, we then obtain from \eqref{eq: er est 1}-\eqref{eq: er est 3}:
	\begin{align*}
		( p - \Pi_{\mathscr{Q}_h} p, &\ D \cdot [\epsilon \bm{v}_h, \hat{\epsilon} \mu_h])_{\Omega} \nonumber\\
		\lesssim&\ \| p - \Pi_{\mathscr{Q}_h} p \|_{L^2(\Omega)} (\| K^{-\frac{1}{2}} \bm{v}_h \|_{L^2(\Omega)} + \| \gamma^{\frac{1}{2}} K_{\nu}^{-\frac{1}{2}} \mu_h^1 \|_{L^2(\Gamma^1)}) \nonumber\\
		\le&\ \| p - \Pi_{\mathscr{Q}_h} p \|_{L^2(\Omega)} 
		(\| K^{-\frac{1}{2}} (\Pi_{\mathscr{V}_h} \bm{u} - \bm{u}_h) \|_{L^2(\Omega)} + \| K^{-\frac{1}{2}} \delta_1 \bm{u}_{p, h} \|_{L^2(\Omega)} \nonumber\\
		&+\| \gamma^{\frac{1}{2}} K_{\nu}^{-\frac{1}{2}} (\Pi_{\Lambda_h}^1 \lambda^1 - \lambda_h^1) \|_{L^2(\Gamma^1)} + \| \gamma^{\frac{1}{2}} K_{\nu}^{-\frac{1}{2}} \delta_1 \lambda_{p, h}^1 \|_{L^2(\Gamma^1)}) \nonumber\\
		\le&\ \left(\frac{1}{2\delta_4} + \frac{1}{2 \delta_5} + 1 \right) \| p - \Pi_{\mathscr{Q}_h} p \|_{L^2(\Omega)}^2 \nonumber\\
		&+ \frac{\delta_4}{2} \| K^{-\frac{1}{2}} \Pi_{\mathscr{V}_h} \bm{u} - \bm{u}_h \|_{L^2(\Omega)}^2 
		+ \frac{1}{2} \delta_1^2 \| K^{-\frac{1}{2}} \bm{u}_{p, h} \|_{L^2(\Omega)}^2 \nonumber \\
		&+ \frac{\delta_5}{2} \| \gamma^{\frac{1}{2}} K_{\nu}^{-\frac{1}{2}} (\Pi_{\Lambda_h}^1 \lambda^1 - \lambda_h^1) \|_{L^2(\Gamma^1)}^2
		+ \frac{1}{2} \delta_1^2 \| \gamma^{\frac{1}{2}} K_{\nu}^{-\frac{1}{2}} \lambda_{p, h}^1 \|_{L^2(\Gamma^1)}^2).
	\end{align*}

	We collect the above and set the parameters $\delta_2$ to $\delta_5$ sufficiently small. In turn, \eqref{eq: u_p,h ineq} and a sufficiently small $\delta_1$ then give us the estimate
	\begin{align*}
		\| K^{-\frac{1}{2}} (\Pi_{\mathscr{V}_h} \bm{u} - \bm{u}_h) \|_{L^2(\Omega)}^2
		&+ \| \gamma^{\frac{1}{2}} K_{\nu}^{-\frac{1}{2}} (\Pi_{\mathit{\Lambda}_h} \lambda - \lambda_h) \|_{L^2(\Gamma)}^2
		+ \| \Pi_{\mathscr{Q}_h} p - p_h \|_{\mathscr{Q}_h}^2 \nonumber\\
		&\lesssim \| [\bm{u} - \Pi_{\mathscr{V}_h} \bm{u}, \lambda - \Pi_{\mathit{\Lambda}_h} \lambda] \|_{\mathscr{X}_h}^2
		+ \| p - \Pi_{\mathscr{Q}_h} p \|_{L^2(\Omega)}^2.
	\end{align*}
	Thus, with \eqref{eq: perfect div}, the triangle inequality, and the properties from \eqref{eqs proj est}, we have shown convergence of order $k+1$ as stated in the following theorem.
	\begin{theorem}[Convergence] \label{thm: convergence}
		Let $(\bm{u}_0, \lambda, p)$ solve \eqref{eq: weakform} and denote $\bm{u} = \bm{u}_0 + \mathcal{R}\lambda$. Analogously, let $(\bm{u}_{0, h}, \lambda_h, p_h)$ solve \eqref{weakform_h} and denote $\bm{u}_h = \bm{u}_{0, h} + \mathcal{R}_h\lambda_h$. Then, given a quasi-uniform grid, the norms from \eqref{Pnorm} and \eqref{norm on u}, and the Fortin interpolators from \eqref{eqs proj est}, the following error estimate holds
		\begin{align}
			\| [\bm{u} - \bm{u}_h, \lambda - \lambda_h] \|_{\mathscr{X}_h}
			+&\ \| p - p_h \|_{\mathscr{Q}_h} \nonumber\\
			&\lesssim \| [\bm{u} - \Pi_{\mathscr{V}_h} \bm{u}, \lambda - \Pi_{\mathit{\Lambda}_h} \lambda] \|_{\mathscr{X}_h}
			+ \| p - \Pi_{\mathscr{Q}_h} p \|_{L^2(\Omega)} \nonumber\\
			&\lesssim h^{k + 1} (\| \bm{u} \|_{k+1,\Omega} 
						+ \| \nabla \cdot \bm{u} \|_{k+1,\Omega} 
						+ \| \lambda \|_{k+1,\Gamma}
						+ \| p \|_{k+1, \Omega}).
		\end{align}
	\end{theorem}

\section{Numerical Results}
\label{sec:numerical_results}

	To confirm the theory derived in the previous sections, we show two sets of numerical results using test cases designed to highlight some of the typical challenges associated with fracture flow simulation. First, we introduce a setup in two dimensions and describe the included features with their associated parameters, followed by an evaluation of the results. This example includes a fracture tip gradually decreasing to zero, thus indicating that \eqref{gamma_0} may not be a necessary condition. Next, a three-dimensional problem is considered which provides an accessible illustration of the dimensional decomposition full dimensional decomposition. 

\subsection{Two-Dimensional Problem}
\label{sub:2D}

	For the two-dimensional test case, the domain $\Omega$ is the unit square. A unit pressure drop is simulated by imposing a Dirichlet boundary condition for the pressure at the top and bottom boundaries of $\Omega$. On the remaining sides, a no-flow boundary condition is imposed. For simplicity, the source function $f$ is set to zero.
	
	Multiple fractures with different properties are added to $\Omega$ to show the extent of the capabilities of the method. Figure~\ref{fig: Domain} (Left) gives an illustration of these fractures. First, the domain contains a fracture intersection. Modeling intersections is non-trivial for a variety of fracture flow schemes and typically calls for special considerations (see e.g. \cite{Formaggia,Scotti_intersect,Sandve2,Schwenck2015}). In contrast, for the method we present, the implementation of an intersection follows naturally due to the dimensional decomposition. Although the test case merely contains a single intersection, this can easily be extended.

	In addition to the intersection, a second aspect is the termination of fractures. The method is designed to handle these endings well, which is investigated by introducing immersed and half-immersed fractures as well as fractures crossing through the boundary as shown in Figure~\ref{fig: Domain} (Left). As suggested in Subsection~\ref{sub:governing_equations}, a fracture ending can either be modeled by ending the feature and setting a zero-flux boundary condition or letting the aperture decrease to zero. Both models are included here, applied to the lower and upper horizontal fractures, respectively.

	By setting the aperture to zero, a virtual extension is created which may be desirable for computational reasons. Due to the close relation to mortar methods, a virtual fracture can serve as an interface between two subdomains with non-matching grids, thus creating a domain decomposition method. By setting the aperture to zero, tangential flow is naturally eliminated and the method simplifies to a mortar scheme with the normal flux as the mortar variable. For our test case, the region where this occurs is illustrated by a dashed line in Figure~\ref{fig: Domain}.

	Furthermore, we investigate the handling of non-matching grids by independently meshing the two sides of all fractures, illustrated by Figure~\ref{fig: Domain} (Right). The mortar mesh is then chosen sufficiently coarse in order to meet requirement \eqref{eq: CR}.

	\begin{figure}[thbp]
		\centering
		\includegraphics[width = 10 cm]{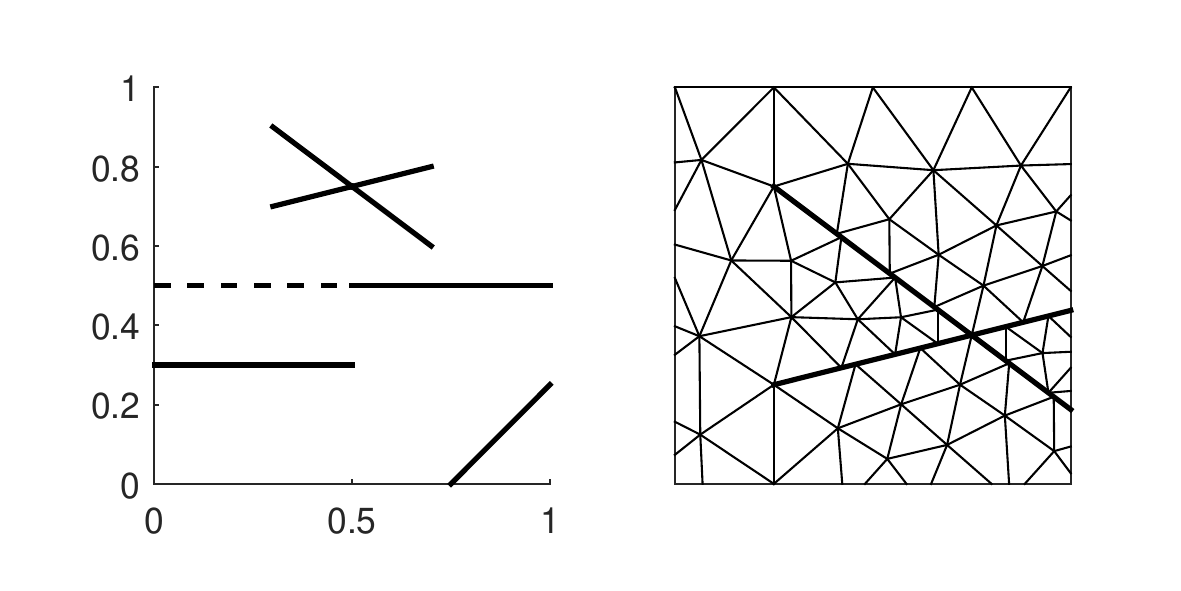}
		\caption{(Left) The domain contains an intersection, multiple fracture endings, a fracture passing through the domain and a virtual extension of a fracture represented by the dashed line. (Right) The grid is non-matching along all fractures, including the sections with zero aperture.}
		\label{fig: Domain}
	\end{figure}

	Let us continue by defining the parameters for the test cases. First, we assume isotropic permeability in $\Omega^2$ and set $K^2$ as the $2 \times 2$ identity tensor. The different included fractures are given different material properties, given in Table~\ref{table:parameters}. 
	The aperture $\gamma$ is chosen as a constant in all fractures except for the central horizontal feature $\Omega_7^1$, which has zero aperture for $x_1 \le 0.5$ and we let the aperture increase for $x_1 > 0.5$ subject to the constraint on the gradient from \eqref{grad eps bound}. The precise formula is given in Table~\ref{table:parameters}.
	Fractures with high permeabilities are expected to stimulate flow whereas a low permeability leads to blocking features. 

	For this example, we assume that $K_\nu^0$, i.e. the permeability in the intersection point, is given. It is possible to define this permeability differently on each interface between fracture and intersection depending on the permeabilities of the attached fractures. Alternatively, a single value can be prescribed, yet this will rely heavily on the modeling assumptions. Here, we omit such procedures in order to present the scheme in the most general setting.

	\begin{table}[thbp]
		\centering 
		\caption{The coordinates and parameters associated with the lower-dimensional domains.}
		\label{table:parameters}
		\begin{footnotesize}\begin{tabular}{ccccl} 
		\hline 
		& $x_{start}$ & $x_{end}$ & $K, K_\nu$ & $\gamma$ \\
		\hline
		$\Omega_1^0$ & (0.5, 0.75) & 			& 100 &0.01 \\[1pt]
		$\Omega_1^1$ & (0.5, 0.75) & (0.7, 0.8) & 100 &0.01 \\[1pt]
		$\Omega_2^1$ & (0.5, 0.75) & (0.3, 0.9) & 100 &0.01 \\[1pt]
		$\Omega_3^1$ & (0.5, 0.75) & (0.3, 0.7) & 100 &0.01 \\[1pt]
		$\Omega_4^1$ & (0.5, 0.75) & (0.7, 0.6) & 100 &0.01 \\[1pt]
		$\Omega_5^1$ & (0.75, 0) & (0.5, 0.75) & 100 &0.01 \\[1pt]
		$\Omega_6^1$ & (0, 0.3) & (0.5, 0.3) & 0.01 &0.01 \\[1pt]
		$\Omega_7^1$ & (0, 0.5) & (1, 0.5) & 0.01 & $0.01(2\max(x_1 - 0.5, 0))^4$ \\
		\hline 
		\end{tabular}\end{footnotesize}
	\end{table}

\subsubsection{Qualitative Results}
\label{sub:qualitative_results}
	
	The results for the two-dimensional test case introduced above are shown in Figure~\ref{fig: Results} with the use of lowest order Raviart-Thomas elements for the flux and piecewise constants for the mortar and pressure variables (see \eqref{Lowest order choice}). As expected, the results are free of oscillations and neither the fracture endings, intersection, or non-matching grid cause problems for the scheme. Moreover, the solution is qualitatively in accordance with the physically expected results.

	Most notably, we observe the effects on the pressure distribution related to the prescribed permeabilities and apertures. High permeabilities enforce a nearly continuous pressure, which is clearly shown both between the fracture and matrix pressures, but similarly between the fracture and intersection pressure, represented by a dot. On the other hand, the two regions with low permeabilities result in a pressure discontinuity across the fracture. Recall that the abrupt fracture ending calls for a no-flux boundary condition, whereas a gradual decrease in aperture naturally stops the flow beyond the closure point of the fracture. From the pressure and flux distributions in Figure~\ref{fig: Results}, it is clear that these two different models for fracture endings lead to different behavior of the solution. In particular, the solution is visibly less regular around the abrupt fracture ending compared to the region where a fracture pinches out. Thus, the result emphasizes the impact of abrupt fracture endings relative to gradual pinch-outs for low permeabilities. 

	As an additional comment, we have also investigated fracture pinch-outs which violate \eqref{grad eps bound}. In this case, minor oscillations are seen near the fracture tip, verifying that inequality \eqref{grad eps bound} is a necessary condition not just for the analysis, but also for the method. 

	\begin{figure}[thbp]
		\centering
		\includegraphics[width = 12 cm]{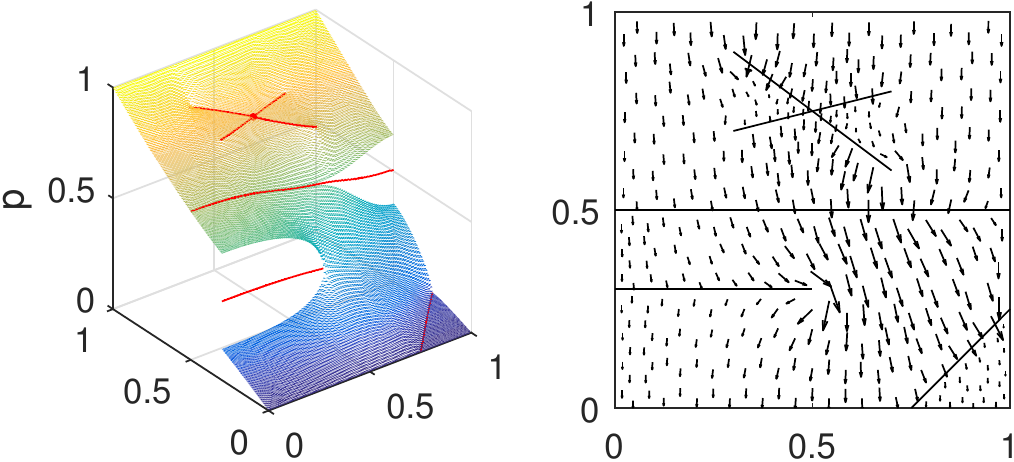}
		\caption{(Left) The pressure distribution for the two-dimensional test case. The effects of abrupt fracture endings as opposed to gradual closure of fractures is apparent around the tips of the blocking features. Continuity of the solution is visible where the aperture equals zero. (Right) The flow uses conducting fractures as preferential flow paths whereas it is forced around the features with low permeability.}
		\label{fig: Results}
	\end{figure} 

\subsubsection{Convergence}
\label{sub:convergence}

	According to the theory, we expect to see linear convergence in all variables for the lowest order choice of spaces described by \eqref{Lowest order choice} with $k = 0$. To verify this, numerical experiments were performed on five consecutively refined grids. All solutions were then compared to the solution on the finest grid.

	Let us continue by describing the norms used in this comparison, starting with the flux variables. These fluxes have irregular behavior around fracture tips resulting in a loss of convergence rates in these regions. For that reason, we exclude balls with some small radius $\rho>0$ centered at the fracture tips, denoted by $\mathfrak{B}_{\rho}$. 	For this test case, it has been found sufficient to set $\rho = 0.02$.
	We emphasize that the flux variable is given by $\bm{u}_h = \bm{u}_{0, h} + \mathcal{R}_h \lambda_h$, i.e. the full flux is compared in accordance with the theory from Subsection~\ref{sub:a_priori_estimates}.
	Moreover, \eqref{eq: perfect div} shows there is no error in the divergence of the flux when comparing the discrete to the continuous solution using the norms from \eqref{norm on u}. Therefore, we consider convergence in the following, appropriately scaled norms:
	\begin{subequations}
		\begin{align}
			| \bm{v} |_{\mathscr{V}} &= \| K^{-1/2} \bm{v} \|_{L^2(\Omega \backslash \mathfrak{B}_{\rho})}, & \bm{v} &\in \mathscr{V}, \\
			|\mu|_{\mathit{\Lambda}} &= \| \gamma^{\frac{1}{2}} K_{\nu}^{-\frac{1}{2}} \mu\|_{L^2(\Gamma)}, & \mu &\in \mathit{\Lambda}, \\
			|q|_{\mathscr{Q}} &= \| \hat{\epsilon}_{\max} q\|_{L^2(\Omega)}, & q &\in \mathscr{Q}.
		\end{align}
	\end{subequations}
	
	The errors and convergence results are shown in Table~\ref{table:Conv}. On average, we observe linear convergence in all variables, confirming the theory. For $d=0$, which corresponds to a point evaluation of the solution, the accuracy becomes dependent on the particular grid near the intersection, and while the general trend is first-order convergence, the particular rates for this example appear erratic. 
	\begin{table}[ht]                                              
	\centering                                                     
	\caption{Relative errors and convergence rates on a grid with typical mesh size $h_{\text{coarse}}$ and consecutively refined grids.}
	\label{table:Conv}                                             
	\begin{footnotesize}
    \begin{tabular}{l|l|cc|cc|cc}
    	\hline
    	& & \multicolumn{2}{c|}{d=0} & \multicolumn{2}{c|}{d=1}
    	 & \multicolumn{2}{c}{d=2} \\
    	& $h/h_{\text{coarse}}$ & Error & Rate & Error & Rate & Error & Rate \\
    	\hline                                                        
    	\multirow{4}{*}{$\bm{u}_h$}		& $2^{0}$ &   &   & 1.40e-01 &   & 1.10e-01 &   \\                 
    			& $2^{-1}$ &   &   & 6.84e-02 & 1.04 & 7.07e-02 & 0.64 \\           
    			& $2^{-2}$ &   &   & 3.17e-02 & 1.11 & 3.19e-02 & 1.15 \\           
    			& $2^{-3}$ &   &   & 1.21e-02 & 1.39 & 1.39e-02 & 1.19 \\  \hline          
    	\multirow{4}{*}{$\lambda_h$}		& $2^{0}$ & 5.46e-02 &   & 1.56e-01 &   &   &   \\                 
    			& $2^{-1}$ & 1.47e-02 & 1.90 & 8.36e-02 & 0.90 &   &   \\           
    			& $2^{-2}$ & 3.74e-03 & 1.97 & 4.32e-02 & 0.95 &   &   \\           
    			& $2^{-3}$ & 1.94e-03 & 0.95 & 2.06e-02 & 1.07 &   &   \\ \hline           
    	\multirow{4}{*}{$p_h$}		& $2^{0}$ & 9.63e-05 &   & 1.04e-02 &   & 2.44e-02 &   \\          
    			& $2^{-1}$ & 4.43e-06 & 4.44 & 4.96e-03 & 1.07 & 1.21e-02 & 1.01 \\ 
    			& $2^{-2}$ & 1.40e-05 & -1.66 & 2.40e-03 & 1.05 & 5.87e-03 & 1.04 \\
    			& $2^{-3}$ & 5.88e-06 & 1.26 & 1.04e-03 & 1.21 & 2.59e-03 & 1.18 \\ 
    	\hline                                                         
	\end{tabular}  \end{footnotesize}                                                
	\end{table} 

	\subsection{Three-Dimensional Problem} 
	\label{sub:three_dimensional_problem}

	The model problem presented in this section is specifically chosen to illustrate the dimensional decomposition in three dimensions. The domain $\Omega$ is constructed by starting with the unit cube and introducing three planar fractures defined by $x_1 = 0.5$, $x_2 = 0.5$, and $x_3 = 0.5$, respectively.

	The dimensional decomposition of $\Omega$ as described in Section~\ref{sub:geometric_representation} is then performed as follows. The fractures split the domain into 8 smaller cubes whose union defines $\Omega^3$. The domain $\Omega^2$ is defined as the union of the fractures excluding the intersection lines (i.e. $\Omega^2$ consists of 12 planes). Next, the union of the 6 intersection lines, after exclusion of the intersection point, forms $\Omega^1$. Finally, the single intersection point with coordinates $(0.5,0.5,0.5)$ defines $\Omega^0$. To conclude, $\Gamma$ is defined as the union of all interfaces between subdomains of codimension one.

	To close the problem, the following boundary conditions are introduced. The pressure is given at the top and bottom by the function $g(\bm{x}) = x_3 \left(x_1^2 + x_2 \right)$. A no-flux condition is set on the remaining boundaries. All fracture planes and lines touching the boundary $\partial \Omega$ naturally inherit these conditions.

	The parameters for this test case are chosen such that the problem reflects conducting fractures. Specifically, we set $K^3 = 1$ as the matrix permeability, $K^d = K_{\nu}^d = 100$ for $0 \le d \le 2$, and the aperture as $\gamma = 0.01$ for all lower-dimensional domains. The simplicial meshes generated for this problem are matching along all intersections and thus, a matching mortar mesh is employed. The discretized problem is implemented with the use of FEniCS \cite{LoggMardalEtAl2012a}.

	Due to the lack of immersed fracture tips in the proposed domain, no special considerations are needed and each variable is expected to converge linearly for all values of $d$. The numerical results displayed by Table~\ref{tab:Rates3DRegular} confirm these expectations. Once again, the solution on a finer grid is used to serve as the true solution.

	\begin{table}[thbp]
		\caption{Relative errors and convergence rates for the 3D problem. The results show that each variable has (at least) first order convergence in each dimension.}
		\label{tab:Rates3DRegular}
		\centering
		\begin{footnotesize}\begin{tabular}{l|l|cc|cc|cc|cc}
		\hline
		& & \multicolumn{2}{c|}{d=0} & \multicolumn{2}{c|}{d=1}
		 & \multicolumn{2}{c|}{d=2} & \multicolumn{2}{c}{d=3} \\
		& h & Error & Rate & Error & Rate & Error & Rate & Error & Rate \\
		\hline
			 			& $2^{-1}$ & &&  1.46e-01 &  &   3.50e-01 &  &   2.76e-01 &  \\
		$\bm{u}_h$	& $2^{-2}$ & &&  4.62e-02 & 1.66 &   1.97e-01 & 0.83 &   1.56e-01 & 0.83 \\
			 			& $2^{-3}$ & &&  1.31e-02 & 1.81 &   9.76e-02 & 1.02 &   7.76e-02 & 1.00 \\
		\hline	
			 			& $2^{-1}$ &   2.24e-01 &  &   2.15e-01 &  &   1.94e-01 &  \\
		$\lambda_h$	 	& $2^{-2}$ &   1.71e-02 & 3.71 &   9.93e-02 & 1.12 &   1.07e-01 & 0.86 \\
			 			& $2^{-3}$ &   5.96e-03 & 1.52 &   4.30e-02 & 1.21 &   5.60e-02 & 0.93 \\
		\hline
			 			& $2^{-1}$ &   4.51e-02 &  &   1.55e-01 &  &   1.50e-01 &  &   1.36e-01 &  \\
		$p_h$	 			& $2^{-2}$ &   7.11e-03 & 2.67 &   7.40e-02 & 1.07 &   7.15e-02 & 1.07 &   6.69e-02 & 1.02 \\
			 			& $2^{-3}$ &   1.49e-03 & 2.25 &   3.29e-02 & 1.17 &   3.17e-02 & 1.17 &   3.06e-02 & 1.13 \\
		\hline
		\end{tabular}\end{footnotesize}
	\end{table}

	To visualize the solution obtained in this test case, Figure~\ref{fig: 3DResults} shows the pressure distribution and the two-dimensional fluxes, i.e. the fluxes tangential to the fractures. Due to the parameters and boundary conditions, the solution exhibits a dominant flow through the conductive fractures from top to bottom. 

	\begin{figure}[thbp]
		\centering
		\includegraphics[width = 0.8\textwidth]{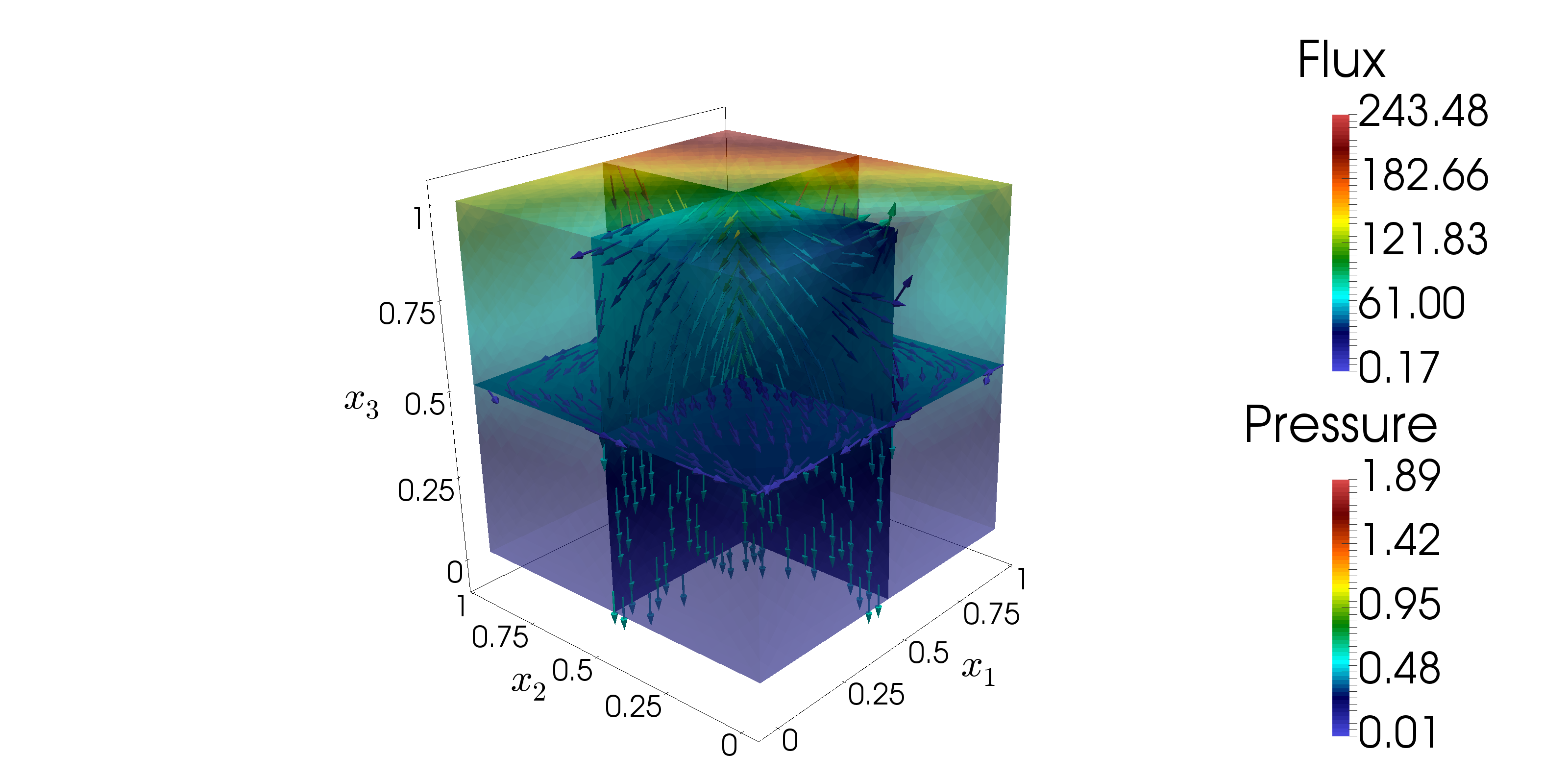}
		\caption{The pressure distribution in the regular three-dimensional case superimposed on the two-dimensional flux field. The solution is qualitatively consistent with expectations for a problem with conducting fractures.}
		\label{fig: 3DResults}
	\end{figure} 

\section{Conclusion}
\label{sec:conclusion}

	In this work, we proposed a mixed finite element method for Darcy flow problems in fractured porous media. The use of flux mortars in a mixed method results in a mass conservative scheme which is able to handle non-matching grids. The key novel components of the method are the hierarchical approach obtained after subdividing the domain in a dimensional manner, as well as the use of dimensionally composite function spaces to analyze the problem with respect to stability and a priori error estimates. Our analysis shows the method is robust and convergent allowing for varying, and arbitrary small apertures. Numerical results confirm the theory, and furthermore show that the constraint on the degeneracy of the normal permeability used in the analysis may not be needed in practice. 

\section*{Acknowledgments}
\label{sec:acknowledgments}

	The authors wish to thank Inga Berre, Sarah Gasda, and Eirik Keilegavlen for valuable comments and discussions on this topic. The research is funded in part by the Norwegian Research Council grants: 233736, 250223. The third author was partially supported by the US DOE grant DE-FG02-04ER25618 and NSF grant DMS 1418947.

\bibliography{reportbib}
	
\end{document}